\documentclass{article}

\usepackage[english]{babel}
\usepackage{latexsym}
\usepackage{amsmath}
\usepackage{amssymb}
\usepackage{graphics}
\usepackage{graphicx}
\usepackage{fancyhdr}
\usepackage{color}
\usepackage{epstopdf}
\usepackage{amsfonts}
\usepackage[utf8]{inputenc}
\usepackage[numbers]{natbib}				
\renewcommand{\cite}[1]{\citeauthor*{#1} [\citeyear{#1}]}
\usepackage{amsthm,url,xspace}
\definecolor{couleurCitations}{rgb}{0,0.65,0}
\definecolor{couleurRef}{rgb}{0.75,0,0}
\usepackage[pdftex,colorlinks=true,
linkcolor=couleurRef,citecolor=blue,bookmarks=true,plainpages=true,
 urlcolor= blue] {hyperref}

\usepackage{float}
\usepackage[caption=false]{subfig}
\usepackage{bm}

\title{\textbf{Title}}                                   %
\makeatletter                                            %
\@addtoreset{equation}{section}                          %
\makeatother                                             %

\usepackage{booktabs}

\usepackage{enumitem}

\usepackage{stackengine}
\setstackEOL{\#}
\setstackgap{L}{.7\baselineskip}

\usepackage{algorithm}
\usepackage[noend]{algpseudocode}
\makeatletter
\def\BState{\State\hskip-\ALG@thistlm}
\makeatother

\makeatletter
\newcommand*{\centerfloat}{%
  \parindent \z@
  \leftskip \z@ \@plus 1fil \@minus \textwidth
  \rightskip\leftskip
  \parfillskip \z@skip}
\makeatother

\usepackage{rotating}

\providecommand{\keywords}[1]{\textbf{Keywords---} \textit{#1}}

\begin{document}

\title{Asymptotic Confidence Regions for Highdimensional Structured Sparsity.}
%
%
%

\author{Benjamin~Stucky\thanks{Seminar for Statistics, ETH Z\"urich, Switzerland}, Sara~van de Geer$^{*}$}
\maketitle

\begin{abstract}
 In the setting of high-dimensional linear regression models, we propose two frameworks for constructing pointwise and group confidence sets for penalized estimators which incorporate
 prior knowledge about the organization of the non-zero coefficients. This is done by desparsifying the estimator as in \citet{sara5} and \citet{sara6}, then using an 
 appropriate estimator for the precision matrix $\Theta$. In order to estimate the precision matrix a corresponding structured matrix norm penalty has to be introduced. 
 After normalization the result is an asymptotic pivot.
 The asymptotic behavior is studied and simulations are added to study the differences between the two schemes.
\end{abstract}



\theoremstyle{plain}
\newtheorem{theorem}{Theorem}
\newtheorem{proposition}{Proposition}
\newtheorem{lemma}{Lemma}
\newtheorem{definition}{Definition}
\newtheorem*{remark}{Remark}
\newtheorem{assumptions}{Assumption}
\newtheorem{properties}{Properties}
\newtheorem{corollary}{Corollary}
\newtheorem{example}{Example}

\newcommand{\op}[1]{\operatorname{#1}}                     
\newcommand{\1}{{\rm 1}\mskip -4,5mu{\rm l} }              
\newcommand{\amin}{\mathop{\mathrm{arg\,min}}}             
\newcommand{\pen}{\op{pen}}                                
\newcommand{\h}[1]{\hat{#1}}                               
\newcommand{\bh}{\hat{\beta}}                               
\newcommand{\bt}{\widetilde{\beta}}                               
\newcommand{\bl}{\bar{\beta}}                               
\newcommand{\bc}{\stackrel{\circ}{\beta}}                               

\newcommand{\eh}{\hat{\epsilon}}                             
\newcommand{\bha}{\hat{\beta}_{\alpha}}                       
\newcommand{\bo}{\beta^{0}}                                    
\newcommand{\Bo}{B^{0}}
\newcommand{\Bh}{\hat{B}}

\newcommand{\0}[1]{{#1}^{0}}                               
\newcommand{\dO}{\delta_{\Omega}}                          
\newcommand{\GO}{\Gamma_{\Omega}}                          
\newcommand{\Obs}{\Omega(\beta_{S})}    
\newcommand{\Rp}{\mathbb{R}^{p}}                           
\newcommand{\R}[1]{\mathbb{R}^{#1}}                        
\newcommand{\Ost}{\Omega_{*}}                              
\newcommand{\N}{\mathcal{N}}                               

\newcommand{\dell}{\delta\left[(\lambda\ltsr{\eh}+\lambda^{S}\ltsr{\epsilon})\Omega(\bh_{S}-\beta)+(\lambda\ltsr{\eh}-\lambda^{S^{c}}\ltsr{\epsilon})\Omega^{S^{c}}(\bh_{S^{c}})\right]}
\newcommand{\delll}{\left[\delta(\lambda+\lambda^{S})\Psi(\Bh_{S}-B_{S})+\delta(\lambda-\lambda^{S^{c}})\Psi^{S^{c}}(\Bh_{S^{c}})\right]}  
\newcommand{\lt}[1]{  \lVert  #1    \rVert_{\ell_{2}}}  
\newcommand{\ltn}[1]{  \lVert  #1    \rVert  _{n}^{2}}            
\newcommand{\ltsr}[1]{ \lVert  #1    \rVert  _{n}}        
\newcommand{\nuc}[1]{\lVert  #1    \rVert  _{nuc}}              
\newcommand{\ltsrr}[1]{\lVert  #1    \rVert  _{n}}              
\newcommand{\slt}[2]{  \langle #1,#2 \rangle}          
\newcommand{\lo}[1]{  \lVert  #1    \rVert  _{1}}            
\newcommand{\linf}[1]{  \lVert  #1    \rVert  _{\infty}}            
\newcommand{\lqq}[1]{  \lVert  #1    \rVert  _{q}}            
\newcommand{\fr}[1]{ \lVert  #1    \rVert  _{F}^{2}}          
\newcommand{\frn}[1]{ \lVert  #1    \rVert  _{F}}          

\newcommand{\Prob}{\mathop{P}}

\newcommand{\supp}{\mathop{\mathrm{supp}}}
\newcommand{\tra}[1]{\operatorname{tr}\left(#1\right)}     
\newcommand{\Epsilon}{\mathcal{E}}

\newcommand{\Rell}{\mathbb{R}}
\newcommand{\Nat}{\mathbb{N}}
\newcommand{\Z}{\mathbb{Z}}
\newcommand{\Q}{\mathbb{Q}}
\newcommand{\Var}{\mathop{Var}}
\newcommand{\Cor}{\mathop{Cor}}
\newcommand{\Cov}{\mathop{Cov}}
\newcommand{\Parcor}{\mathop{Parcor}}
\newcommand{\peff}{\mathrm{peff}}
\newcommand{\sign}{\mathop{sign}}
\newcommand{\argmax}{\mathop{\arg \max}\limits}
\newcommand{\argmin}{\mathop{\arg \min}\limits}
\newcommand{\eps}       {\varepsilon}
\newcommand{\PP} {{  \rm I\hskip-0.22em P}}
\newcommand{\bx}{{\bf X}}
\newcommand{\bu}{{\bf U}}
\newcommand{\by}{{\bf Y}}
\newcommand{\bz}{{\bf Z}}
\newcommand{\cb}{{\cal B}}
\newcommand{\ch}{{\cal H}}
\renewcommand{\bm}{{\cal B}_m}
\newcommand{\cf}{{\cal F}}
\newcommand{\cs}{{\cal S}}
\newcommand{\hcs}{\hat{\cal S}}
\newcommand{\cR}{{\cal R}}
\newcommand{\EE} {{\rm I\hskip-0.48em E}}
\newcommand{\ind}{{\rm I\hskip-0.58em 1}}

\algnewcommand\INPUT{\item[\textbf{Input:}]}%
\algnewcommand\OUTPUT{\item[\textbf{Output:}]}%

\newtheorem*{thetwo}{Theorem 8}

\newsavebox\MBox
\newcommand\Cline[2][red]{{\sbox\MBox{$#2$}%
  \rlap{\usebox\MBox}\color{#1}\rule[-2.1\dp\MBox]{\wd\MBox}{1.3pt}}}

    %
\selectlanguage{english}  %

\section{Introduction}
We focus on the basic high dimensional linear regression model, which is at the core of understanding more complex models:
\begin{equation}\label{modell}Y=X\beta^{0}+\epsilon.\end{equation}
Here $Y \in \mathbb{R}^{n}$ is an observable response variable, $X$ is a given $n \times p$ design matrix with $p>>n$,
$\beta^{0}\in \mathbb{R}^{p}$ is a parameter vector of unknown coefficients and $\epsilon\in \mathbb{R}^{n}$ is unobservable noise.
Due to the high-dimensionality of the design the question arises as to find the solution to an underdetermined system. 
The idea to restrict ourselves to sparse solutions has become the new paradigm to solve this
problem for high-dimensional data.
In such a setting, the LASSO estimator (introduced by \citet{tib1}) is the most widely used method in 
pursuance of estimating the unknown parameter vector $\beta^{0}$, 
while avoiding the high-dimensional problem of overfitting:
$$\hat{\beta}_{\ell_{1}}:=\amin_{\beta\in\mathbb{R}^{p}}\left\{\ltn{Y-X\beta}+2\lambda_{L}\lo{\beta} \right\}.$$
The loss function is defined as $\ltn{Y-X\beta}:=\sum_{j =1}^{n}(Y-X\beta)_{j}^{2}/n$, and $\lo{\beta}$ denotes the $\ell_{1}$-norm. 
The main purpose of the added $\ell_{1}$-norm penalty is to achieve an entry-wise sparse $\hat{\beta}_{\ell_{1}}$ solution, while at the same time the least squares 
loss ensures good prediction properties. 
Furthermore the constant $\lambda_{L}>0$ is the penalty level, regulating the amount of sparsity introduced to the solution.

The $\ell_{1}$-norm penalty is a simple convex relaxation of the non-convex $\ell_{0}$ penalty ($\lVert\beta \rVert_{\ell_{0}}:= \# \{i:\beta_{i}\neq 0\}$).
Let us recall that the $\ell_{1}$-norm penalty does not promote any specific sparsity structure. 
In other words the LASSO estimator does not assume anything about the organization of the non-zero coefficients.
In this sense the LASSO estimator does not incorporate any prior knowledge of the structure of the true unknown active set $S_{0}:=\{i: \beta_{0,i}\neq 0\}$. 
In practice however, prior knowledge is often available. Prior knowledge may emerge from physical systems or known biological processes. 
For the purpose of integrating the available prior information, the $\ell_{1}$-norm penalty needs to be replaced in such a way, that the new penalty reflects this knowledge. 
One can find many examples of such penalties and their
properties in the recently emerging literature on the sparsity structure of the unknown parameter vector, see for example \citet{bach1}, \citet{bach2}, \citet{mi1}, \citet{mi2}, \citet{mi3}. 
A more comprehensive overview can be found in \citet{obo1}.

We will focus on norm penalties and therefore generalize the LASSO estimator to a large family of penalized estimators (see \citet{sara1} and \citet{benji1}), 
each with distinct properties to promote sparsity structures in the parameter vector:
\begin{equation} \label{sromega}\bh_{\Omega}:=\amin_{\beta\in\mathbb{R}^{p}}\left\{\ltn{Y-X\beta}+\lambda\Omega(\beta)\right\}.\end{equation}
Here $\Omega$ is any norm on $\mathbb{R}^{p}$ that reflects some aspects of the pattern of sparsity for the parameter vector $\beta^{0}$. Again for readability we let $\bh=\bh_{\Omega}$. 
We characterize the $\Omega$-norm in terms of its weakly decomposable subsets of $\mathbb{R}^{p}$. 
A weakly decomposable norm is in some sense
able to split up into two norms, one norm measuring the size of the vector on the active set and the other norm the size on its complement. 
The weakly decomposable norm itself reflects the prior information of the underlying sparsity.
\\
{\bf Notation:} Depending on the context, for a set $J \subset \{ 1 , \ldots , p \}$ and
a vector $\beta \in \mathbb{R}^{p}$ the vector $\beta_J$ is either the $|J|$-dimensional vector
$\{ \beta_j:\ j \in J\}$ or the $p$-dimensional vector $\{ \beta_j {\rm l} \{ j \in J \} : \ j=1 , \ldots , p \} $.
More generally, for a vector $w_J := \{ w_j :\ j \in J\}$, we use the same
notation for its extended version $w_J \in \mathbb{R}^{p}$ where $w_{j, J} =0 $ for all $j \notin J$. 
For a set ${\cal B}$ we let ${\cal B}_J= \{ \beta_J: \ \beta \in {\cal B} \} $.\\

The definition of a weakly decomposable norm is crucial to the following sections, so we introduce it as in \citet{sara1} or \citet{benji1}. This idea goes back to \citet{bach2}.
\begin{definition}[Weak decomposability]
 A norm $\Omega$ in $\mathbb{R}^{p}$ is called weakly decomposable for an index set $S\subset\{1,...,p\}$, if
 there exists another norm $\Omega^{S^{c}}$ on $\mathbb{R}^{|S^{c}|}$ such that
 \begin{equation}\label{weak}\forall \beta\in\mathbb{R}^{p}: \ \  \Omega(\beta_{S})+\Omega^{S^{c}}(\beta_{S^{c}})\leq \Omega(\beta).\end{equation}
\end{definition}
A set $S$ is called allowed if $\Omega$ is a weakly decomposable norm for this set.
From now on we use the notation $\Upsilon_{S}(\beta):=\Omega(\beta_{S})+\Omega^{S^{c}}(\beta_{S^{c}})$ the lower bounding norm from the weak decomposability definition and 
$\Lambda_{S}(\beta):=\Omega(\beta_{S})+\Omega(\beta_{S^{c}})$ the upper bounding norm from the triangle inequality. 
The weak decomposability now reads
$$\Upsilon_{S}(\beta)\leq\Omega(\beta)\leq \Lambda_{S}(\beta).$$
Therefore 
the $\Upsilon_{S}$-norm mimics the decomposability property of the $\ell_{1}$-norm $(\lo{\beta}=\lo{\beta_{S}}+\lo{\beta_{S^{c}}})$ for the set $S$.\\

For the LASSO estimator, most work up until recently has been focusing on point estimation among other topics, with not much focus on establishing uncertainty in high dimensional models. 
Interest has been growing rapidly on the very important topic of constructing confidence regions for the LASSO estimator,
see for example \citet{sara5}, \citet{sara6}, \citet{zhang}, \citet{monte} and \citet{nico1}. When it comes to confidence regions for structured sparsity estimators there has not yet been done much work to our knowledge.
The paper \citet{sara6} mentions one approach for group confidence regions for structured sparsity briefly, which we will develop further.

The main goal of this paper is therefore to construct asymptotic group confidence regions for structured sparsity estimators in two possible ways.
In order to do this, we introduce a de-sparsified version of the estimators in \eqref{sromega}, following the idea of \citet{sara5}. An appropriate estimation of the precision matrix will
be needed for the definition of a de-sparsified estimator. The estimation of the precision matrix can be done in two ways which are beneficial for the construction of asymptotic confidence regions.
These two frameworks differ in the structure of the penalty function. The theoretical behavior and the assumptions on the sparsity is studied. 
Furthermore, a simulation compares these two frameworks in the high dimensional case and outlines potential applications.

\section{De-sparsified $\Omega$ structured estimator}
For a given norm $\Omega(\cdot)$ on $\mathbb{R}^{p}$ we can determine its sparsity structure by listing all the subsets $\mathfrak{S}:=\{S_{1},...,S_{k}\}$ for which the norm is weakly decomposable. 
The estimator $\bh_{\Omega}$ \eqref{sromega} prefers to set the complement of any of the sets $S_{1}^{c},...,S_{k}^{c}$ to zero. Unfortunately the joint distribution of estimator \eqref{sromega}
is not easy to access. 
But it is possible to de-sparsify \eqref{sromega} and asymptotically describe the distribution of this new estimator. The essential idea for the de-sparsified estimator comes from the following 
lemma, which establishes a variation to the KKT conditions of $\bh_{\Omega}$, following directly from \citet{benji1}.
\begin{lemma}\label{t1}
 For the estimator defined in \eqref{sromega} with $\eh:=Y-X\bh$ the KKT conditions are
 $X^{T}\eh/n=\lambda\hat{Z},$
 where $\Omega^{*}(\hat{Z})\leq 1$ and $\hat{Z}^{T}\bh=\Omega(\bh)$.
\end{lemma}
Here $\Omega^{*}(\cdot)$ is another norm on $\mathbb{R}^{p}$ called the dual norm.
$$\Omega^{*}(\alpha):=\sup_{\beta\in\mathbb{R}^{p}, \Omega(\beta)=1}\beta^{T}\alpha.$$
Since $Y=X\beta^{0}+\epsilon$ and using the notation $\hat{\Sigma}:=X^{T}X/n$ we can write the KKT conditions as
$$\hat{\Sigma}(\bh-\beta^{0})+\lambda\hat{Z}=X^{T}\epsilon/n.$$
Suppose we have an appropriate surrogate for the precision matrix $\hat{\Theta}$, we get
$$\hat{\Theta}\hat{\Sigma}(\bh-\beta^{0})+\lambda\hat{\Theta}\hat{Z}=\hat{\Theta}X^{T}\epsilon/n, \text{ and}$$
$$\bh+\lambda\hat{\Theta}\hat{Z}-\beta^{0} = \hat{\Theta}X^{T}\epsilon/n + \Delta/\sqrt{n}$$
Here $\Delta:= \sqrt{n}(\hat{\Theta}\hat{\Sigma}-I)(\bh-\beta^{0})$ is the error term.
We define the de-sparsified $\Omega$ structured estimator as follows.
\begin{definition}
 The de-sparsified $\Omega$ structured estimator is 
 $$\hat{b}_{\Omega}:=\bh+\lambda\hat{\Theta}\hat{Z}.$$
\end{definition}

When $\Omega(\cdot)=\lo{\cdot}$ is the $\ell_{1}$-norm, and if we have a $\beta_{\ell_{1}}$ sparsity assumption of order $o(\sqrt{n}/\log(p))$, a reasonable sparsity 
assumption on the precision matrix 
and if we assume the errors to follow i.i.d. Gaussian distributions, then 
\citet{sara5} have shown that the de-sparsified $\ell_{1}$ structured estimator follows an asymptotic Gaussian distribution with an asymptotically negligible error term.

In order to get similar results for the $\Omega$ penalization, we need to discuss how to estimate the precision matrix $\hat{\Theta}$.
The main problem that arises is, that good estimation error bounds are only available expressed in the $\Upsilon_{S_{\star}}$-norm, where $S_{\star}$ is the unknown
oracle set from the main theorem in \citet{benji1}.
The next two sections give two different ways to estimate $\hat{\Theta}$ in
such a way that $\Delta$ is asymptotically negligible.


\section{First framework: gauge confidence regions}\label{sec:g}
A way to construct an estimate for the precision matrix is to 
do $|J|$-wise regression with any fixed set $J\subset \{1,...,p\}$. $|J|$-wise regression is a very similar method as node-wise regression (introduced by \citet{nico2}), 
but instead of one node, we have simultaneously $|J|$ nodes. With this $|J|$-wise regression we try to capture the group interdependencies stored in the precision matrix.
This is why we require a multivariate model of the form
\begin{equation}\label{Jwise}\hat{B}_{J}:=\amin_{B_{J}\in \mathbb{R}^{|J^{c}|\times |J|}}\left(\nuc{X_{J}-X_{J^{c}}B_{J}}+\lambda_{J}\Psi(B_{J})\right).\end{equation}
The nuclear norm is defined as
$$\nuc{A}:= \tra{\sqrt{A^{T}A}}=\sum_{i=1}^{\min(n,|J|)}\sigma_{i}(A),$$ 
where $\sigma_{i}(A)$ are the singular values of a $n\times |J|$ matrix $A$ and for a square $m\times m$ matrix $B$
the trace function is defined as $\tra{B}:=\sum_{i=1}^{m}B_{i,i}$.
Furthermore the penalty is defined as 
\begin{equation}\label{psi1}\Psi(A):=\sum\limits_{j=1}^{|J|} g(A_{j}).\end{equation}
It is a matrix norm on $\mathbb{R}^{|J^{c}|\times|J|}$ (it is the dual matrix norm of an operator norm), that uses the 
computational cost effective $\ell_{1}$-norm on the columns together with another norm $g$ on $\mathbb{R}^{p}$.
Here $A_{j}\in\mathbb{R}^{p}$ is equal to the $j$-th column of the matrix $A$ on the set $J^{c}$, and $0$ on the set $J$. 
The norm $g$ is defined so that it lower bounds all $\Upsilon_{S}$-norms where $S\in \mathfrak{S}$ is any non trivial allowed set of the $\Omega$-norm. Furthermore the norm $g$ should satisfy 
the following reflection property
$$g(\beta_{f(J)})=g(\beta)\text{ }\text{ } \text{, where }\beta_{f(J)}:=\beta_{J}-\beta_{J^{c}}.$$
This is a natural condition on $g$, because for each allowed set $S$ we have
$$\Upsilon_{S}(\beta_{f(S)})=\Upsilon_{S}(\beta).$$
Where $\Upsilon_{S}$ is defined as $\Upsilon_{S}(\beta):= \Omega(\beta_{S})+\Omega^{S^{c}}(\beta_{S^{c}})$.
In order to construct the norm $g$ we construct a convex set where we will take the gauge function. Remark that $\min_{S\in\mathfrak{S}}\Upsilon_{S}(\cdot)$ is in general not a norm, 
therefore we need to take the convex hull. The convex
set is defined through 
$$\overline{B}:=\bigcup\limits_{S\text{ allowed}}B_{\Upsilon_{S}}, \text{ with } B_{\Upsilon_{S}}\text{ the unit ball of } \Upsilon_{S}-\text{norm, }$$
$$B_{g}:=\operatorname{Conv}\left(\overline{B}\cup \operatorname{flip}_{J}(\overline{B}) \right).$$
The function $\operatorname{flip}_{J}(\cdot)$ reflects a set along the hyperplane defined by the subset $J$. To be more precise for a subset $B\subset\mathbb{R}^{p}$ we have
$$\operatorname{flip}_{J}(B):=\{\gamma : \gamma_{J}=\beta_{J} \text{ and } \gamma_{J^{c}}=-\beta_{J^{c}}, \forall \beta\in B \}.$$
Then we can define its gauge function (also known as Minkowski functional) as follows:
$$g(x):=\inf\left(\lambda>0; \text{ }x\in \lambda B_{g}\right).$$
See Figure \ref{3fig1} for a 
graphical representation of the gauge function.
\begin{figure}[ht]
\centering
  \raisebox{-0.5\height}{\includegraphics[width = 6cm]{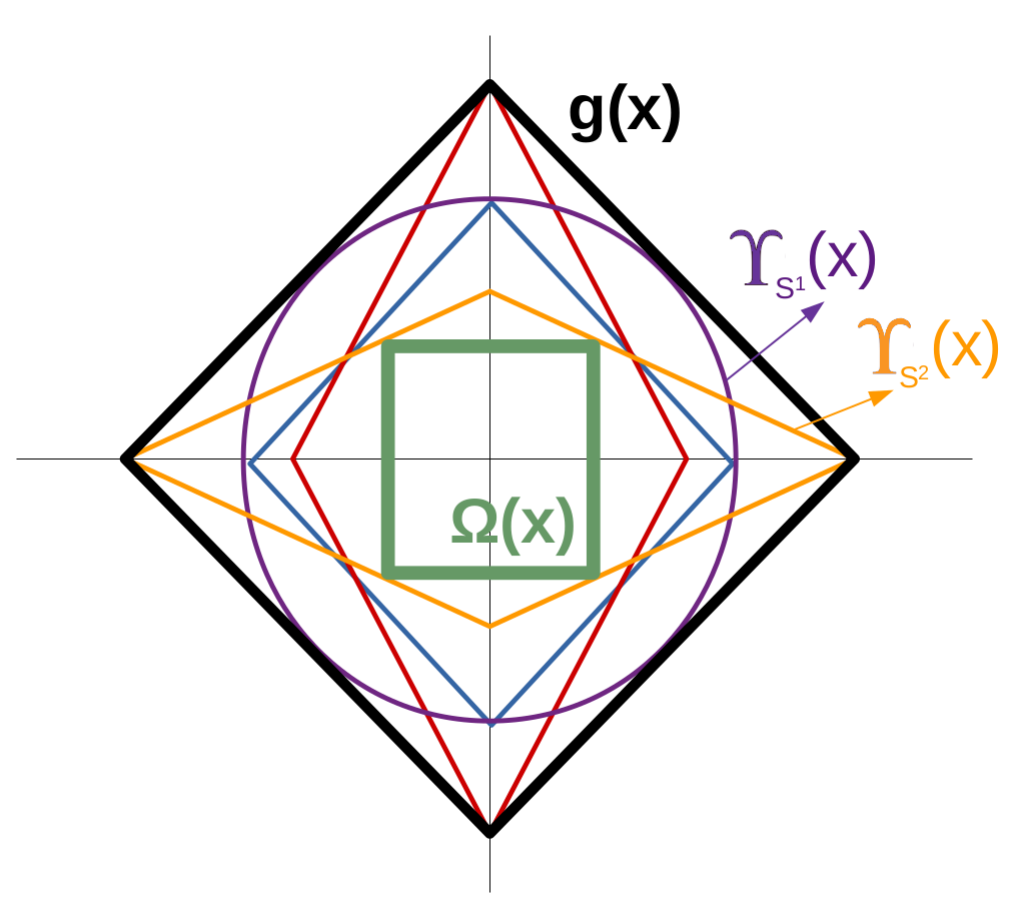}}
  \raisebox{-0.5\height}{\includegraphics[width = 6cm ]{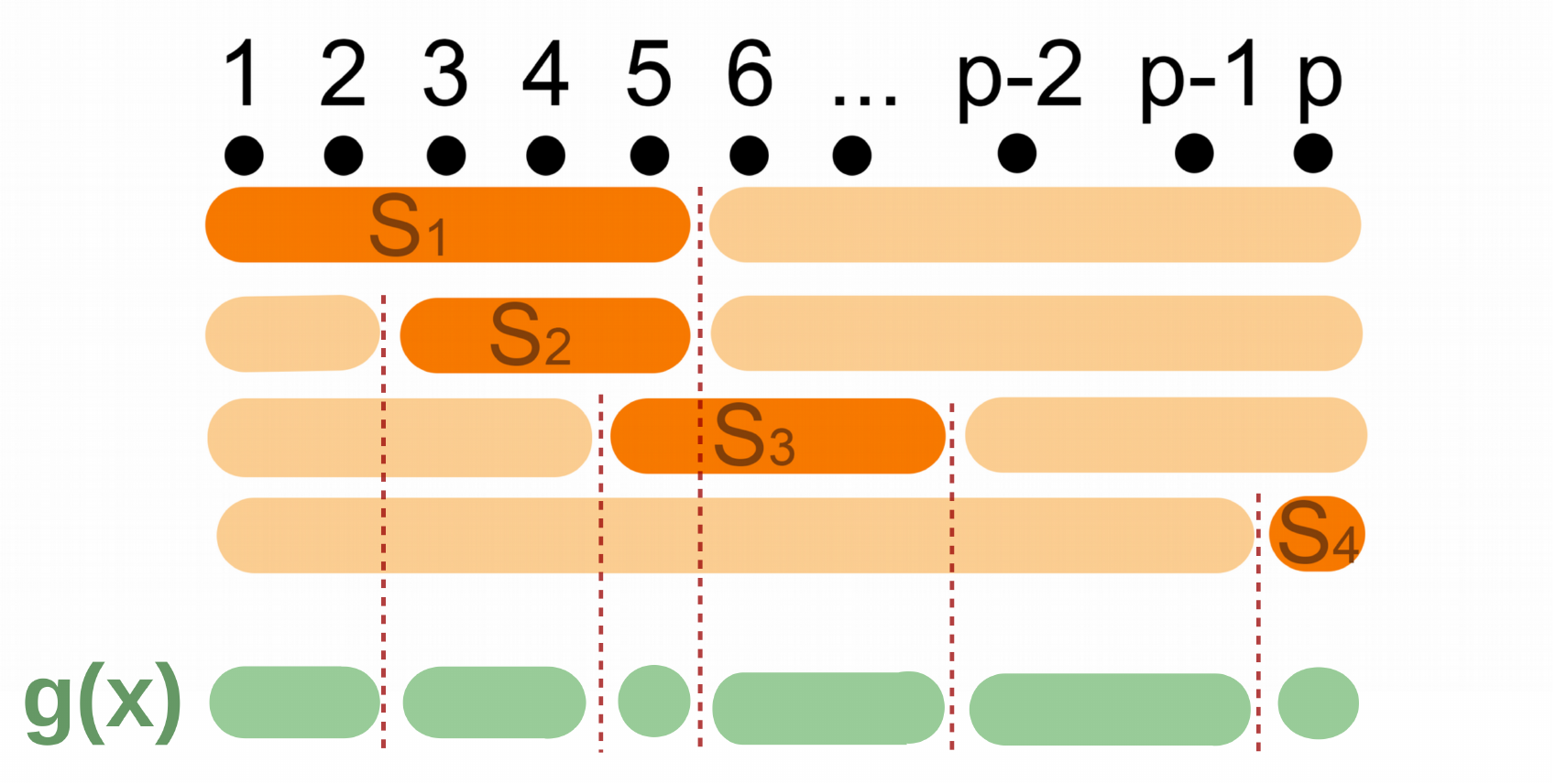}}
\caption{Intuition about the gauge function. Left: lower bounding nature, right: additive nature.}
\label{3fig1}
\end{figure}
From this definition of the $g$ we can see that the following Lemma holds.

\begin{lemma}\label{glemma}
  For the gauge function $g$ the following properties hold
 \begin{enumerate}
  \item $g$ defines a norm on $\mathbb{R}^{p}$.
  \item $g(\beta)\leq \Upsilon_{S}(\beta) \leq \Omega(\beta)$ and $g(\beta)\leq\Upsilon_{S}(\beta_{f(J)}), \text{ }\forall S$ allowed sets.
  \item $g(\beta)= \left(\max\limits_{S \text{ allowed}}\max\left(\Upsilon^{*}_{S}(\beta),\Upsilon^{*}_{S}(\beta_{f(J)})\right)\right)^{*}$.\\
  Furthermore $\Upsilon^{*}_{S}(z)=\max\left(\Omega^{*}(\beta_{S}),\Omega^{S^{c,*}}(\beta_{S^{c}})\right)$ $\forall\beta\in\mathbb{R}^{p}$.
  \item $g(\beta_{J^{c}})\leq g(\beta)$ for all $\beta\in \mathbb{R}^{p}$.
 \end{enumerate}
\end{lemma}
Lemma \ref{glemma} covers the main properties of the gauge function. Result (1) shows that $\Psi$ is in fact a matrix norm. Result (3) gives a characterization of what the gauge function is in our case.
Results (2) and (4) are the main properties of the function $g$, which will be needed to let the error term for the de-sparsified estimator go to 0.
Why we chose to construct the $\Psi$-norm for the $|J|$-wise regression as a column sum of this gauge function $g$ will become more evident later in this paper.
Regarding the construction of confidence sets by means of estimating the precision matrix through the $|J|$-wise multivariate regression, we will need to specify the 
Karush-Kuhn-Tucker (KKT) conditions for
the multivariate regression estimator $\hat{B}_J$.
The first thing we will need is the subdifferential of a matrix norm.
In the paper of \citet{watson} one can find the formulation of the subdifferential for a norm of a $m\times n$ matrix $A$
$$\partial ||A||=\left\{G\in \mathbb{R}^{m\times n}: ||B||\geq ||A||+ \tra{(B-A)^{T}G}, \text{ } \forall \text{ }B\in \mathbb{R}^{m\times n}\right\}.$$
Furthermore we have the following characterization of the subdifferential
\begin{equation}\label{matrixsub}G\in \partial ||A|| \Longleftrightarrow \begin{cases} i) \text{ }\text{ }||A||\text{ }= \tra{G^{T}A}\\ ii)\text{ } ||G||^{*}\leq 1 \end{cases}.\end{equation}
Here the dual matrix norm is defined as $||A||^{*}=\sup\limits_{B: \text{ }||B||\leq 1}\tra{B^{T}A}$. 
Let us briefly note that, by definition of the dual matrix norm, a generalized version of the Cauchy Schwartz Inequality holds true for matrices
$$\tra{AB^{T}}\leq ||A||\cdot ||B||^{*}.$$
Therefore the dual of the $\Psi$-matrix norm is defined as
$$\Psi^{*}(A):=\sup\limits_{B: \text{ }\Psi(B)\leq 1}\tra{B^{T}A}.$$
Applying \vspace*{0.1cm}equation \eqref{matrixsub} to the optimal solution of equation \eqref{Jwise} \vspace*{0.1cm}leads to the KKT conditions (in the case of $\nuc{X_{J}-X_{J^{c}}\hat{B}_{J}}\neq 0$):
\begin{align}\hat{B}_{J} \text{ is optimal }&\Longleftrightarrow \frac{1}{\lambda_{J}}X_{J^{c}}^{T}(X_{J}-X_{J^{c}}\hat{B}_{J})\hat{\Sigma}_{J}^{-1/2}/n\in \partial\Psi(\hat{B}_{J}) \nonumber    \\
&\hspace*{-2cm}\Longleftrightarrow \begin{cases}i)\text{ }\text{ }\lambda_{J}\Psi(\hat{B}_{J})= \tra{\left\{X_{J^{c}}^{T}(X_{J}-X_{J^{c}}\hat{B}_{J})\hat{\Sigma}_{J}^{-1/2}/n\right\}^{T}\hat{B}_{J}} \\ 
ii)\text{ }\lambda_{J}\geq \Psi^{*}\left(X_{J^{c}}^{T}(X_{J}-X_{J^{c}}\hat{B}_{J})\hat{\Sigma}_{J}^{-1/2}/n\right) .\end{cases}
\end{align}
Here we denote $\hat{\Sigma}_{J}:= (X_{J}-X_{J^{c}}\hat{B}_{J})^{T}(X_{J}-X_{J^{c}}\hat{B}_{J})/n$ (assumed to be non-singular). Let us 
additionally define the $|J|$ de-sparsified $\Omega$ structured estimator with the help of the following notations. 
$$T_{J}:=(X_{J}-X_{J^{c}}\hat{B}_{J})X_{J}/n.$$
The normalizing matrix can then be written as
$$M:=\sqrt{n}\hat{\Sigma}_{J}^{-1/2}T_{J}$$
This leads to the definition of the $|J|$ de-sparsified $\Omega$ structured estimator. Defining a de-sparsified estimator in this way 
lets us deal with group-wise confidence sets.
\begin{definition}
 The $|J|$ de-sparsified $\Omega$ structured estimator is 
 \begin{equation}\label{Jdesp}\hat{b}_{J}:= \bh_{J}+T_{J}^{-1}(X_{J}-X_{J^{c}}\hat{B}_{J})^{T}(Y-X\bh)/n.\end{equation}
\end{definition}
With these definitions we are now ready to describe the asymptotic behavior of the estimator \eqref{Jdesp} in the following Theorem.
\begin{theorem}\label{th1}
 Assume that the error in the model \eqref{modell} is i.i.d. Gaussian distributed $\epsilon\sim\mathcal{N}_{n}(0,\sigma_{0}^{2}I)$. Then with the definition
 $\hat{b}_{J}$ from \eqref{Jdesp} together with $\hat{B}_{J}$ as an estimator of the precision matrix and its normalized version $M\hat{b}_{J}$ we have
 $$M(\hat{b}_{J}-\beta_{J}^{0})/\sigma_{0}=\mathcal{N}_{|J|}(0,I)+rem,$$
 where the $\ell_{\infty}$ norm of the reminder term $rem$ can be upper bounded by
 \begin{align}\linf{rem}&\leq \sqrt{n}\lambda g(\hat{\beta}_{J^{c}}-\beta^{0}_{J^{c}})/\sigma_{0}\nonumber\\
 &\leq \sqrt{n}\lambda \Upsilon_{S_{\star}}(\hat{\beta}-\beta^{0})/\sigma_{0}.\end{align}
\end{theorem}
As we can see from Lemma \ref{glemma} (4) we can upper bound part of the reminder term from Theorem \ref{th1} as
$g(\hat{\beta}_{J^{c}}-\beta^{0}_{J^{c}})\leq g(\hat{\beta}-\beta^{0}).$
By the definition of the gauge function $g$, from Lemma \ref{glemma} (2) we get that
$$g(\hat{\beta}_{J^{c}}-\beta^{0}_{J^{c}})\leq \Upsilon_{S}(\hat{\beta}-\beta^{0}),\text{ }\forall\text{ }S\text{ allowed sets of }\Omega.$$
But how can we bound $\Upsilon_{S_{\star}}(\hat{\beta}-\beta^{0})$? From \citet{sara1} and \citet{benji1} we can get sharp oracle results for an estimation error expressed 
in a measure very close to the 
$\Upsilon_{S_{\star}}$-norm, where $S_{\star}$ is the active set of the oracle, but the used measure is not quite the $\Upsilon_{S_{\star}}$-norm. 
A refined version of the theorem in \citet{sara1} leads to sharp oracle result, which we will use to upper bound $\linf{rem}$.
%
%
The Lemma \ref{th11} can be found in the Appendix. In conclusion 
Theorem \ref{th1} together with Lemma \ref{th11} leads to the asymptotic normality of the normalized de-sparsified $\Omega$
estimator on the set $J$. A studentized version leads to an asymptotic pivot. To get the studentized version one could for example use \citet{benji1}
or generalize the more optimal bounds from the paper \citet{sara6}.
The results are summarized in the following corollary.

\begin{corollary}\label{cor1}
 Assume that the error in the model \eqref{modell} is i.i.d. Gaussian distributed $\epsilon\sim\mathcal{N}_{n}(0,\sigma_{0}^{2}I)$. The de-sparsified estimator $\hat{b}_{J}$ is as in  
 \eqref{Jdesp} and the normalized version $M\hat{b}_{J^{c}}$, together with the multivariate estimator $\hat{B}_{J}$ from \eqref{psi1}, as an estimator of the precision matrix.
 Assume that  $0\leq \delta< 1$, and also that $\lambda^{m}< c\lambda_{J}$, with $\lambda^{m}$ as in Lemma \ref{th11}. Furthermore assume $\sqrt{n}\lambda_{J}\zeta\longrightarrow 0$ as $n\to \infty$.
 We invoke weak decomposability for $S$ and $\Omega$. Assume we have a consistent estimator of $\sigma_{0}$. Then we have
 $$\lVert M(\hat{b}_{J}-\beta^{0}_{J})\rVert_{\ell_{2}}^{2}/\hat{\sigma} = \chi_{|J|}^2 ( 1+ o_{\PP }(1)).$$
\end{corollary}
With Corollary \ref{cor1} asymptotic confidence sets can be constructed. But the size of the set $J$ is not controlled. One can find an approach with the group LASSO and the nuclear norm as a penalty in 
\citet{Mitra}, but they need more assumptions. We only need to assume the usual sparsity assumptions on $\beta^{0}$, we do not assume sparsity on $X$. It just happens, due to the
KKT conditions, that a sparse surrogate of the precision matrix bounds the remainder term.

\section{Second framework: $\Omega$ confidence sets}\label{sec:o}
The first framework made use of the gauge function $g$, which is able to lower bound all the $\Upsilon_{S}$-norms associated with the $\Omega$-norm, therefore the remainder term was
asymptotically negligible. But here we will discuss a more direct approach in order to estimate the precision matrix with the $\Omega$-norm itself. But there might be a price to pay. 
This approach was discussed briefly in \citet{sara6} but without mentioning the full consequences of this approach.
In contrast to the first framework, $J$ needs to be a non trivial allowed set of $\Omega$ (complements of allowed sets would also work). 
It is quite natural to be interested in allowed sets (or complements of it). We define another multivariate optimization procedure to get an approximation of the precision matrix as
\begin{equation}\label{xi1}\hat{C}_{J}:=\amin_{C_{J}\in \mathbb{R}^{|J^{c}|\times |J|}}\left(\nuc{X_{J}-X_{J^{c}}C_{J}}+\lambda_{J}\Xi(C_{J})\right).\end{equation}
Here we again use the nuclear norm for its nice KKT properties together with the following norm
\begin{equation}\label{xi}\Xi(A):=\sum\limits_{j=1}^{|J|} \Omega(A_{j}).\end{equation}
In fact we use the $\Omega$-norm as a measure of the columns of a $|J^{c}|\times|J|$ matrix $A$, where $A_{j}$ denotes again the $j$-th column of the matrix $A$ on the set $J^{c}$, 
and $0$ on the set $J$. One new problem arises in this setting, namely that for all allowed sets $S$
$$\Upsilon_{S}(\beta)\leq \Omega(\beta), \text{ }\forall\text{ }\beta\in\mathbb{R}^{p}.$$
Therefore some work has to be done in order to get good bounds for the estimation error expressed in the $\Omega$-norm. 
And this is why we will need to modify the sparsity assumption in order for the reminder term of a de-sparsified version of $\bh_{J}$ to be asymptotically negligible.
\begin{lemma}\label{th21}
 For any weakly decomposable norm $\Omega$ there exists a constant $C_{S_{\star}}$ which may depend on the support $S_{\star}$ of the true underlying parameter $\beta^{0}$ such that
 $$\Omega(\beta_{S_{\star}^{c}})\leq C_{S_{\star}} \Omega^{S_{\star}^{c}}(\beta_{S_{\star}^{c}}), \forall \beta\in\mathbb{R}^{p}.$$
 Here $S_{\star}$ denotes again the optimal allowed oracle set from Lemma \ref{th11}.
\end{lemma}
This means that we need to quantify how far off the $\Upsilon_{S_{\star}}$-norm on $S_{\star}^{c}$ is compared to the $\Omega$ norm, see Figure \ref{csfig}.
\begin{figure}\centering
 \includegraphics[width=0.9\textwidth]{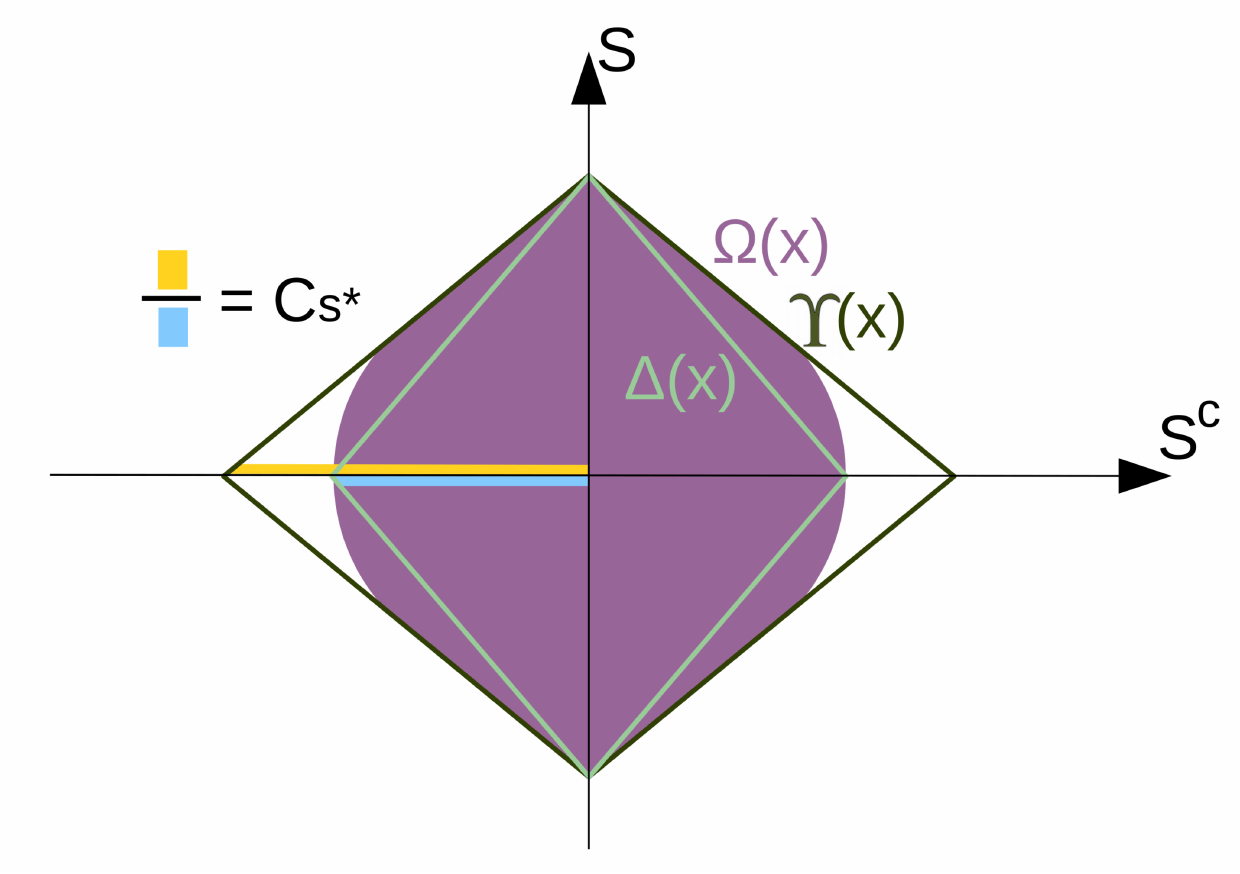}
 \caption{Intuition about the constant $C_{S_{\star}}$.} \label{csfig}
\end{figure}

For the estimation error expressed in the $\Omega$-norm one can already find oracle results in the literature. One can see for example the consistency result Proposition 6 in \citet{obo1}.
But the result from Lemma \ref{th21} together with Lemma \ref{th11} provides more optimal results for our case. This is due to the fact, that the sub optimal constant $1/\rho$ from \citet{obo1} 
appears squared in the bound. Therefore our constant is better suited for our problem. In the Section \ref{sec:ex} we will further discuss this for some widely used examples and show how to choose the
constant $C_{S_{\star}}$ for those examples.

Again, as in Section \ref{sec:g} we need to define a de-sparsified version of the estimator $\bh$. 
This will be a different de-sparsified estimator due to a different estimation of the precision matrix.
In a similar fashion to Section \ref{sec:g} we have the following definitions
\begin{align*}
T_{J}&:=(X_{J}-X_{J^{c}}\hat{C}_{J})^{T}X_{J}/n\\
\hat{\Sigma}_{J}&:= (X_{J}-X_{J^{c}}\hat{C}_{J})^{T}(X_{J}-X_{J^{c}}\hat{C}_{J})/n.
\end{align*}
For the sake of simplicity and readability we keep the same notations as in Section \ref{sec:g} for all these definitions, even though they are defined through $\hat{C}_{J}$ and not $\hat{B}_{J}$. 
\begin{definition}
 The $|J|$ de-sparsified $\Omega$ estimator is again defined as 
 \begin{equation}\label{des2}
 \hat{b}_{J}:= \bh_{J}+T_{J}^{-1}(X_{J}-X_{J^{c}}\hat{C}_{J})^{T}(Y-X\bh)/n.
 \end{equation}
 Here $M:=\sqrt{n}\hat{\Sigma}^{-1/2}T_{J}$ and the normalized version of $\hat{b}_{J}$ is
 $M\hat{b}_{J}.$
\end{definition}

Now with the help of Lemma \ref{th21} we can formulate the following theorem.

\begin{theorem}\label{3th2}
 Assume that the error in the model \eqref{modell} is i.i.d. Gaussian distributed $\epsilon\sim\mathcal{N}_{n}(0,\sigma_{0}^{2}I)$. Then with the definition
 $\hat{b}_{J}$ from \eqref{des2} together with $\hat{B}_{J}$ as an estimator of the precision matrix and its normalized version $M\hat{b}_{J}$ we have
 $$M(\hat{b}_{J}-\beta_{J}^{0})=\mathcal{N}_{|J|}(0,I)+rem.$$
 Where the $\ell_{\infty}$ norm of the reminder term $rem$ can be upper bounded by
 $$\linf{rem}\leq 2\sqrt{n}\lambda C_{S_{\star}}\Upsilon_{S_{\star}}(\hat{\beta}-\beta^{0}).$$
\end{theorem}
Again a similar corollary to Corollary \ref{cor1} holds for this construction of confidence regions, but with an additional sparsity assumption. This sparsity assumption needs
to be specified case by case. It depends on the $\Omega$-norm.

\section{Examples of penalties and their behavior in the two frameworks}\label{sec:ex}

\begin{table}[p]
   \begin{center}\centerfloat\rotatebox{90}{
    \resizebox{1.55\textwidth}{!}{
     \begin{tabular}{ l  p{5.6 cm} l  p{5.6 cm}}
     \toprule
      &$\ell_{1}$-norm& &Lorentz Norm \\ 
     \cmidrule(lr){2-2}\cmidrule(lr){4-4}
     \raisebox{-\totalheight}{\includegraphics[width=0.17\textwidth]{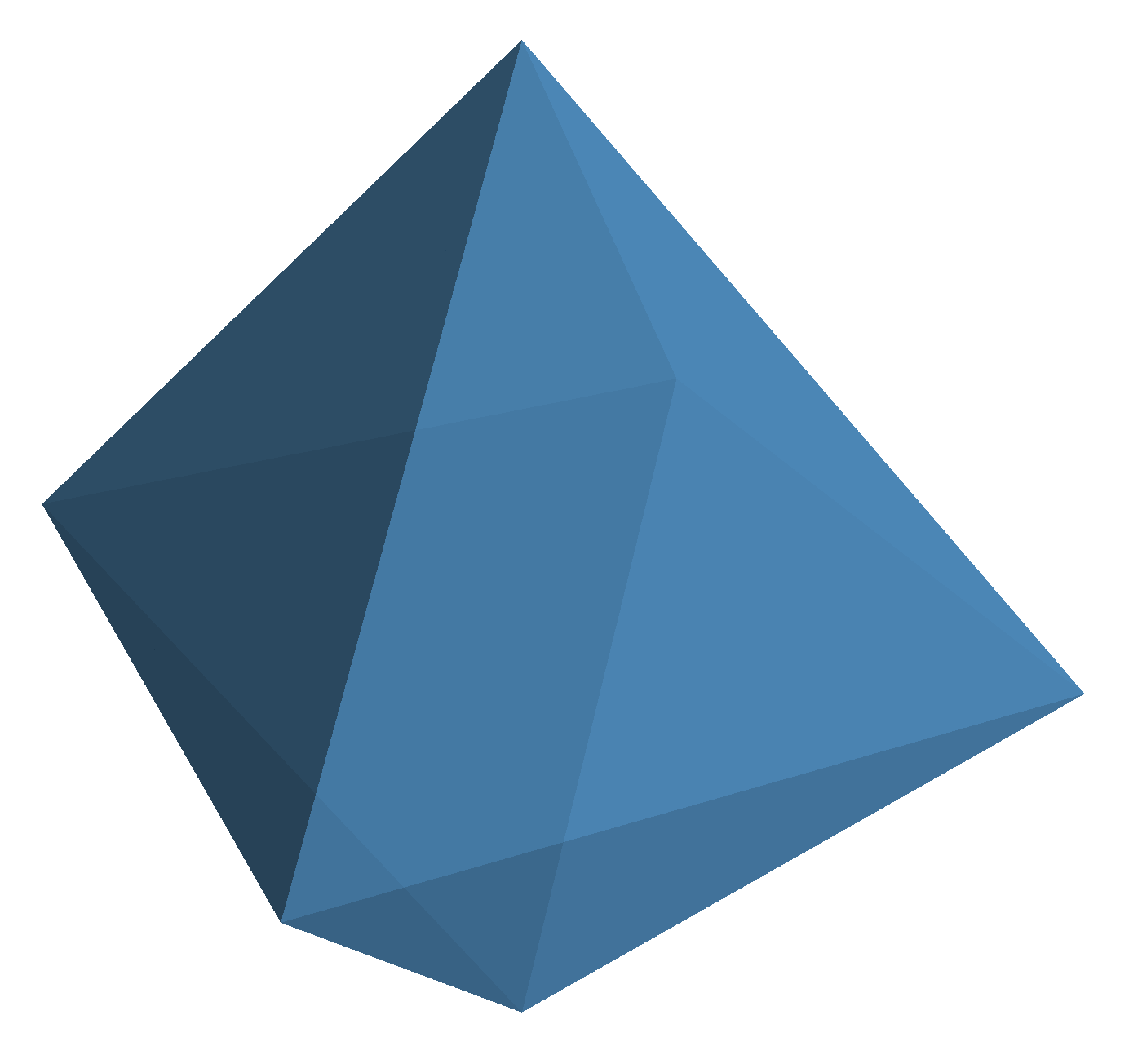}}\hspace{-6mm}
      & \vspace{-1mm}
      \begin{itemize}[leftmargin=*]\footnotesize
      \item[] All subsets $S\subset\{1,...,p\}$ are allowed.
      \item[] $\Omega^{S^{c}}(\beta_{S^{c}})=\lo{\beta_{S^{c}}}$.
      \item[] $g(\beta):=\lo{\beta}$ and $C^{S_{\star}}:=1$.
      \end{itemize}
      &\raisebox{-\totalheight}{\includegraphics[width=0.17\textwidth]{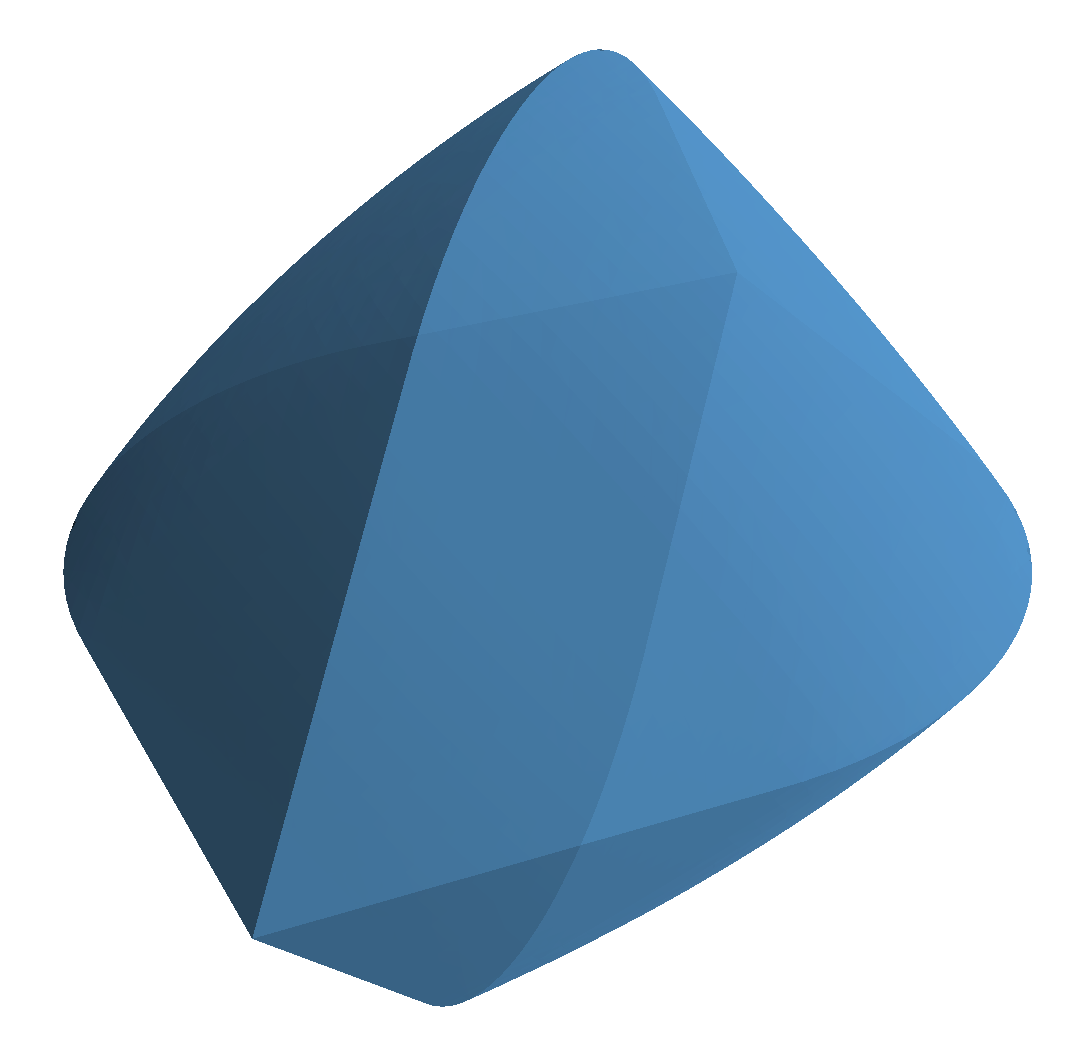}}\hspace{-5mm}
      &\vspace{-1mm}
      \begin{itemize}[leftmargin=*]\footnotesize
      \item[] $S=\left\{p, S_{-p}\right\}$, with $S_{-p}\subset\{1,...,p-1\}$.
      \item[] \scriptsize$\Omega^{S^{c}}(\beta_{S^{c}})=\min\limits_{a_{S^{c}}\in\mathcal{A}_{S^{c}}}\frac{1}{2}\left(\sum\limits_{j\in S^{c}}\frac{\beta_{j}^{2}}{a_{j}}+a_{j}\right).$\footnotesize
      \item[] $g(\cdot)=\lo{\cdot}$ and $C^{S_{\star}}:=3/2$.
      \end{itemize}

      \\
      &Group LASSO norm&& Wedge Norm\\
     \cmidrule(lr){2-2}\cmidrule(lr){4-4}
     \raisebox{-\totalheight}{\includegraphics[width=0.19\textwidth]{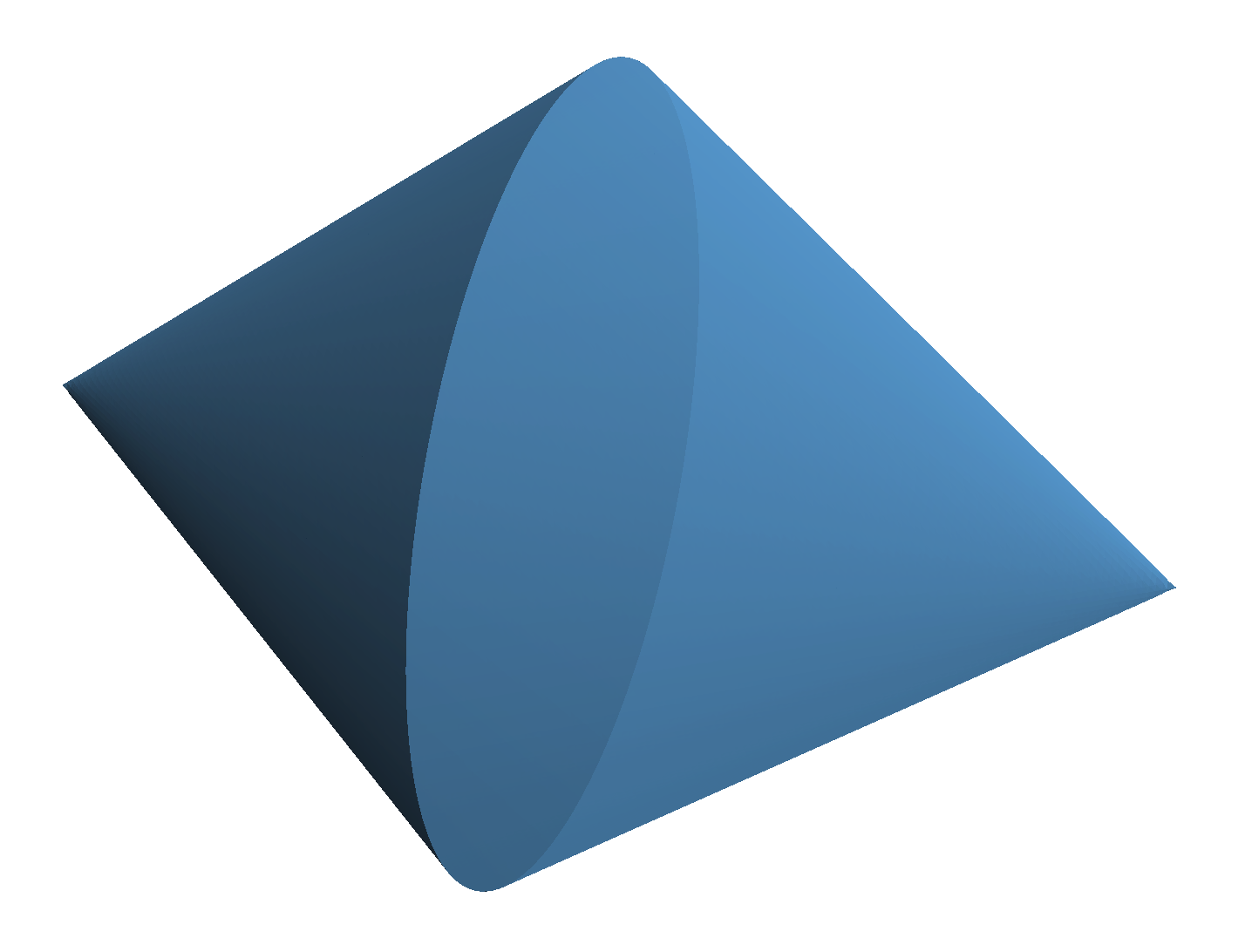}}\hspace{-5mm}
      &\vspace{-1mm}
      \begin{itemize}[leftmargin=*]\footnotesize
      \item[] All subsets consisting of groups $S=\cup_{j\in\mathcal{J}}G_{j},$ $\text{ }\text{ }\mathcal{J}\subset\{1,...,g\}$ are allowed.
      \item[] $\Omega^{S^{c}}(\beta_{S^{c}})=\lVert \beta_{S^{c}}\rVert_{grL}$.
      \item[] $g(\beta):=\lVert \beta \rVert_{grL}$ and $C^{S_{\star}}:=1$.
      \end{itemize}
      &\raisebox{-\totalheight}{\includegraphics[width=0.19\textwidth]{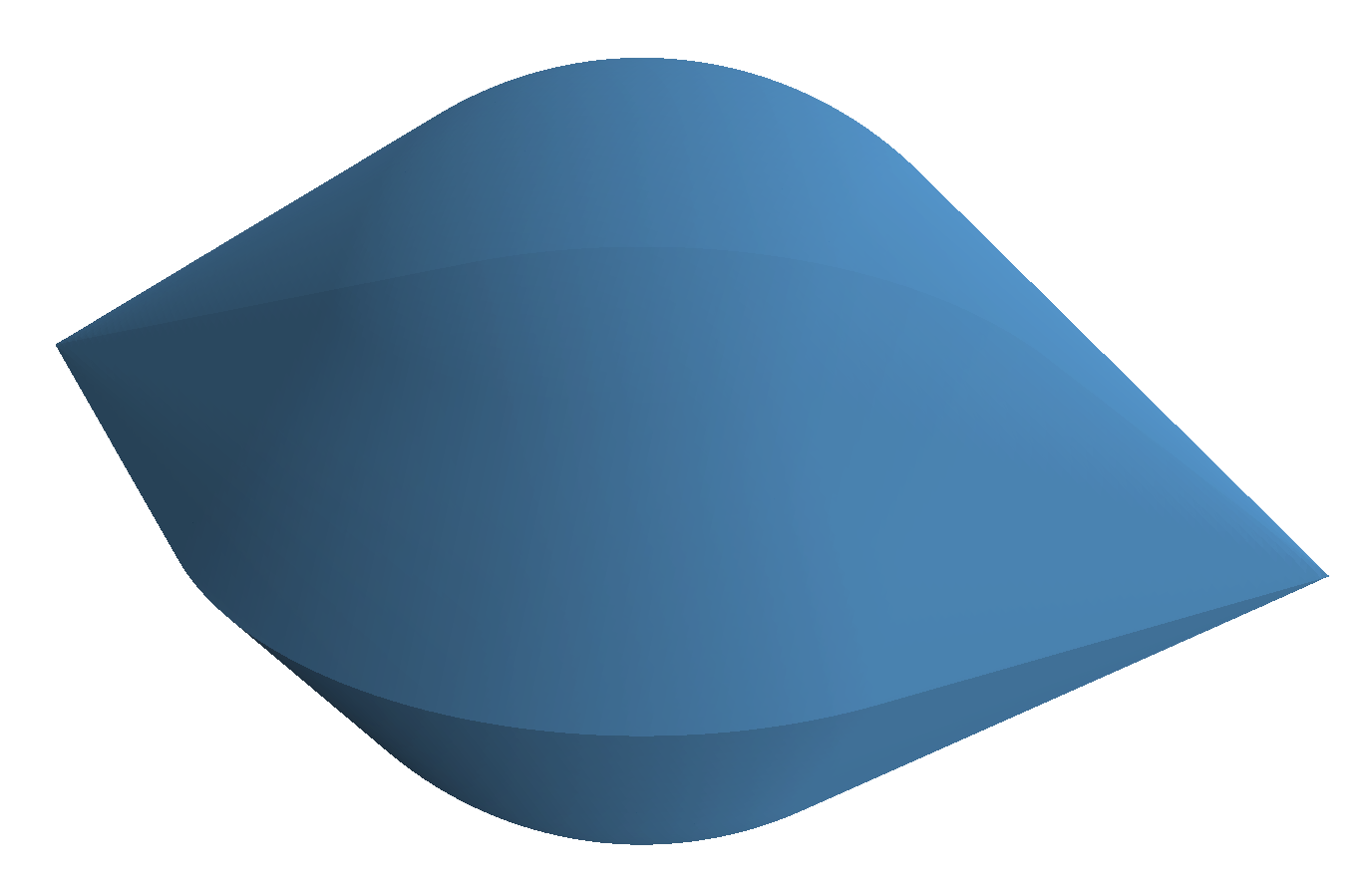}}\hspace{-5mm}
      &\vspace{-1mm}
      \begin{itemize}[leftmargin=*]\footnotesize
      \item[] All sets of the form $S=\{1,...,s\}$, with some $1\leq s\leq p$.
      \item[] $\Omega^{S^{c}}(\beta_{S^{c}})=\Omega(\beta_{S^{c}},\mathcal{A}_{S^{c}})$.
      \item[] $g(\beta):=\lo{\beta}$ and $C^{S_{\star}}:=\sqrt{|S_{\star}|+1}$.
      \end{itemize}

      \\
      &Weighted $\ell_{1}$-norm&&Group Wedge Norm\\
     \cmidrule(lr){2-2}\cmidrule(lr){4-4}
     \raisebox{-\totalheight}{\includegraphics[width=0.17\textwidth]{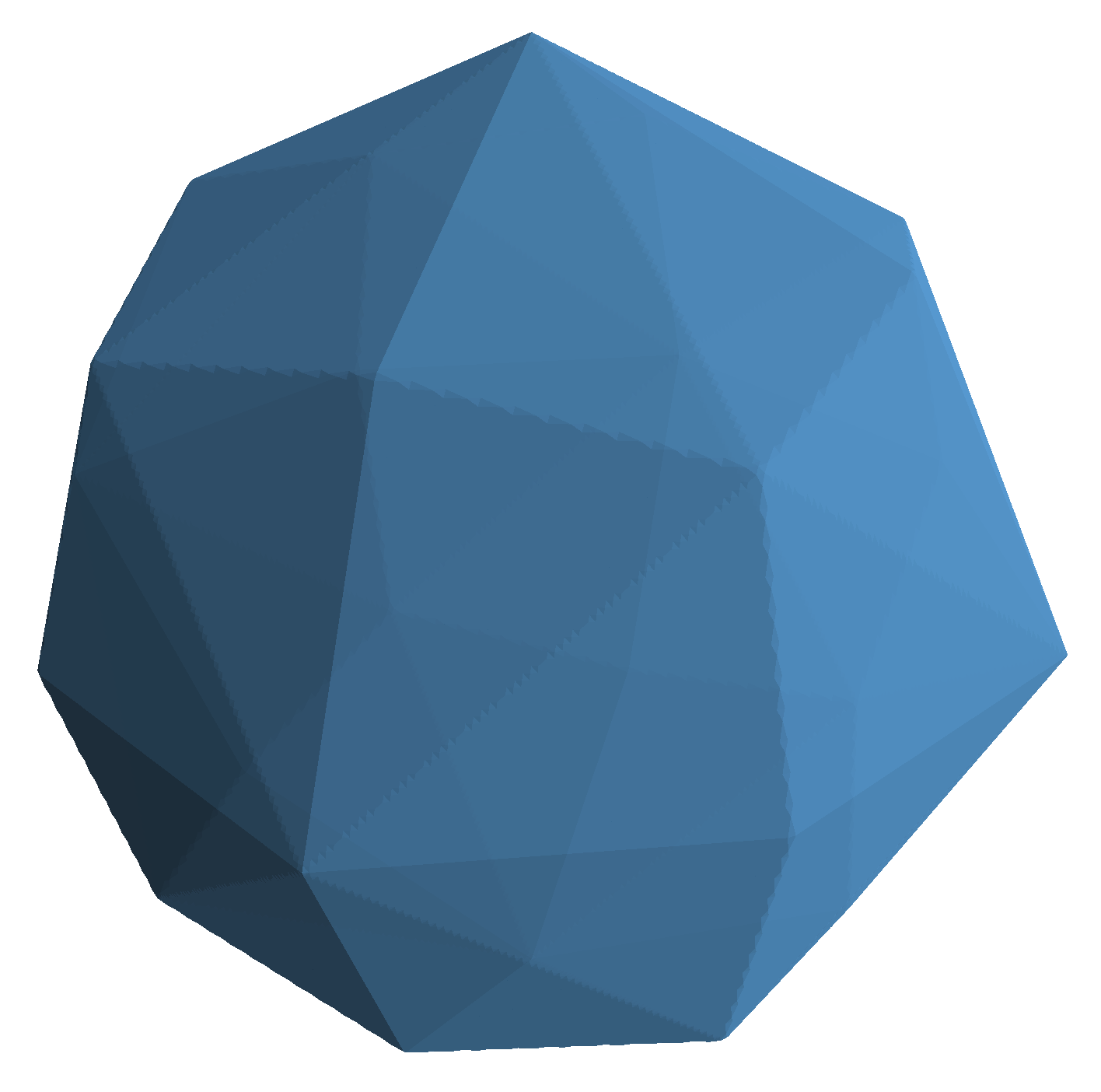}}\hspace{-5mm}
      &\vspace{-1mm}
      \begin{itemize}[leftmargin=*]\footnotesize
      \item[] All subsets consisting of groups $S=\cup_{j\in\mathcal{J}}G_{j}, \text{ }\mathcal{J}\subset\{1,...,g\}$ are allowed.
      \item[] $\Omega^{S^{c}}(\beta_{S^{c}})=\sum_{i=1}^{|S^c|}l_{|S|+i}|\beta|_{(i,S^{c})}$.
      \item[] $g(\beta):=l_{p}\lo{\beta}$ and $C^{S_{\star}}:=\frac{l_{1}}{l_{p}}=o(\log(p))$.
      \end{itemize}
      &\raisebox{-\totalheight}{\includegraphics[width=0.19\textwidth]{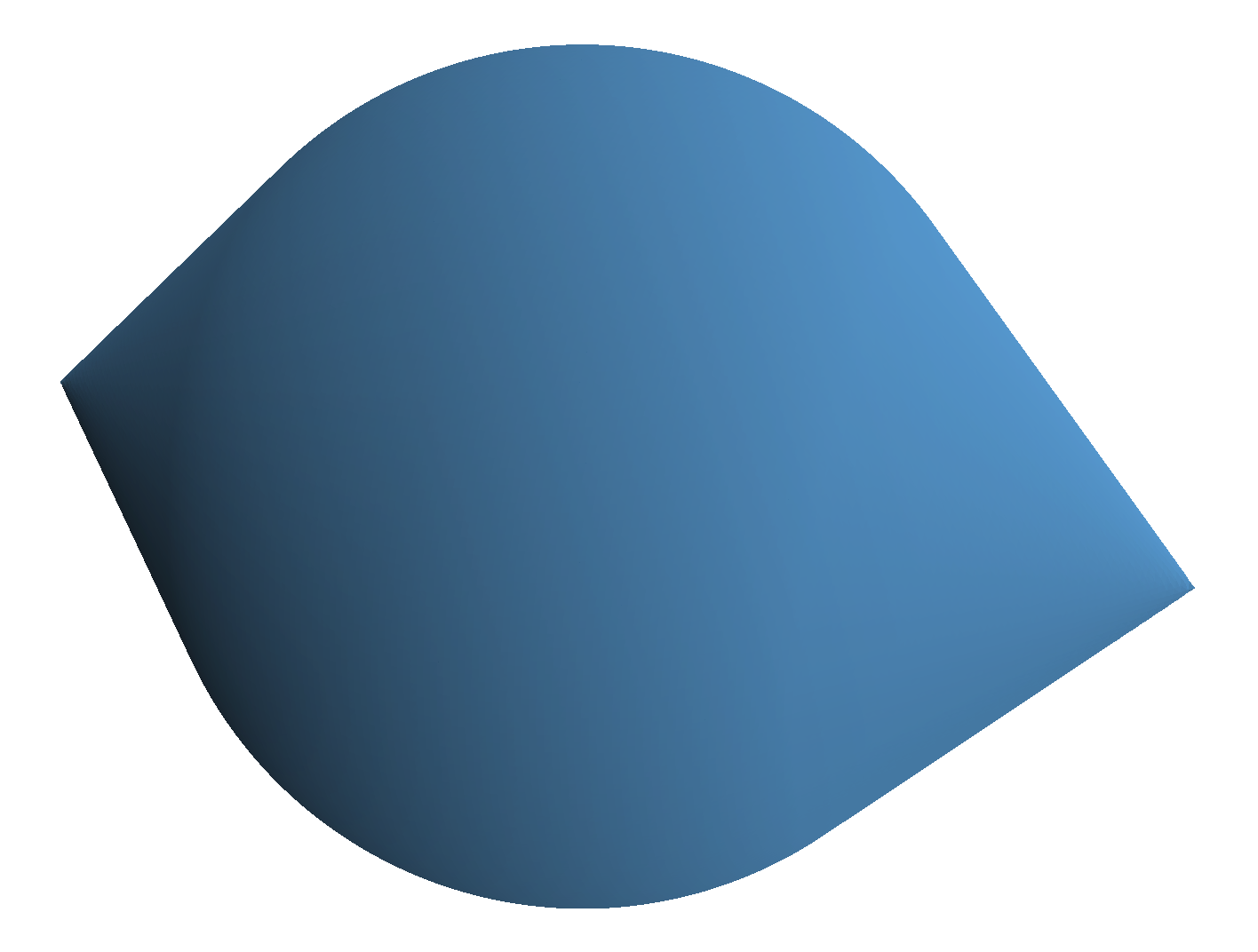}} \hspace{-5mm}
      & \vspace{-1mm}
      \begin{itemize}[leftmargin=*]\footnotesize
      \item[] All sets of the form $S=\{G_{1},...,G_{s}\}$, with some $1\leq s\leq g$.
      \item[] \tiny $\Omega^{S^{c}}(\beta_{S^{c}})=\left\lVert(\lVert\beta_{G_{|S|+1}}\rVert_{2},...,\lVert\beta_{G_{g}}\rVert_{2})^{T}\right\rVert_{W}$.\footnotesize
      \item[] $g(\beta):=\left\lVert\beta^{\mathcal{G}}\right\rVert_{grL}$ and \newline$C^{S_{\star}}:=\sqrt{|S_{\star}|+1}$.
      \end{itemize}

      \\ \bottomrule
      \end{tabular}} }
      \vspace{-3mm}
       \caption{Summary of norm properties.} 
      \label{tbl1}      
  \end{center}
\end{table}
\vspace{1cm}
In this section we try to give the gauge functions $g$ and the constant $C^{*}$ for some of the common norm penalties used in the literature and for some interesting new norm penalties.
Furthermore Table \ref{tbl1} gives an overview of the properties of each example.

\subsection{LASSO: the $\ell_{1}$ Penalty}

As already mentioned the weak decomposable norms all collapse into the $\ell_{1}$-norm due its decomposability. Therefore the gauge function is 
$g(\beta)=\lo{\beta}.$
This means that both of the frameworks for constructing asymptotic confidence sets are in fact the same.
Indeed $\ell_{1}$ has a constant of 
$C_{S_{\star}}=1.$

\subsection{Group LASSO}
The Group LASSO norm is defined by $\lVert \beta \rVert_{grL}:=\sum_{i=1}^{g}\lt{\beta_{G_{i}}}$, where $\{G_{1},...,G_{g}\}$ is a partition of $\{1,...,p\}$.
We know that the active sets for this norm are the groups themselves $S=\cup_{i\in S_{g}} G_{i}$ where $S_{g}$ is any subset of $\{1,...,g\}$.
The gauge function is the group LASSO itself
$g(\beta)=\sum_{i=1}^{g}\lt{\beta_{G_{i}}}.$

Due to the nested $\ell_{1}$-nature of the group LASSO penalty, we have similar decomposable properties as the $\ell_{1}$-norm and get
$C_{S_{\star}}=1.$

\subsection{SLOPE}
The sorted $\ell_1$ norm together with some decreasing sequence $1\geq l_{1}\geq l_{2}\geq...\geq l_{p}> 0$ is defined as
 $$J_{l}(\beta):=l_{1}\lvert\beta\rvert_{(1)}+...+l_{p}\lvert\beta\rvert_{(p)}.$$
 This was shown to be a norm by \citet{zeng}. The SLOPE was introduced by \citet{candes1} in order to control the false discovery rate:
 $$\hat{\beta}_{SLOPE}:=\arg\min_{\beta\in\mathbb{R}^{p}}\left\{\ltn{Y-X\beta}+\lambda J_{l}(\beta)\right\}.$$
 For the SLOPE we have the following two lemmas.
\begin{lemma}\label{lslope}
 For the SLOPE $g(\beta)=l_{p}\lo{\beta}$.
\end{lemma}

\begin{lemma}\label{lslope2}
 The SLOPE has $C_{S_{\star}}= l_{1}/l_{p}$.
\end{lemma}

\subsection{Wedge}
The wedge norm was introduced in \citet{mi2}, and fits in a more broader structured sparsity concept. This concept is nicely compatible from the viewpoint of
weakly decomposable norms, as discussed at length in \citet{sara1}. Let us define the convex cone $\mathcal{A}:=\{a:a\in\mathbb{R}^{p}_{++},a_{j}\geq a_{j+1}, j\in \mathbb{N}_{n-1}\}$, 
where $\mathbb{R}^{p}_{++}$ denotes the positive orthant. Then
the wedge norm is defined as
 $$\lVert \beta\rVert_{W}=\Omega(\beta;\mathcal{A}):= \min_{a\in \mathcal{A}}\frac{1}{2}\sum\limits_{j=1}^{p}\left(\frac{\beta_{j}^{2}}{a_{j}}+a_{j}\right),$$
 with the notation $0/0=0$.
 Define
 $$\mathcal{A}_{S}:=\{a_{S}: a\in \mathcal{A}\}.$$
 Moreover, \citet{sara1} showed that any $S$ satisfying $\mathcal{A}_{S}\subset\mathcal{A}$ is an allowed set for the wedge norm 
 with $\Omega^{S^{c}}(\beta_{S^{c}}):=\Omega(\beta_{S^{c}},\mathcal{A}_{S^{c}})$. This leads to $S:=\{1,..,s\}$ for any $s\in\{1,...,p-1\}$ being an allowed set.
 Hence the wedge estimator can be defined as
 $$\hat{\beta}_{Wedge}=\amin_{\beta \in \mathbb{R}^{p}} \left\{ \ltn{Y-X\beta }+\lambda \lVert \beta\rVert_{W}\right\}.$$
\begin{lemma}\label{wedge1}
 For the wedge norm the gauge function is the $\ell_{1}$-norm $g(\beta)=\lo{\beta}$, for all $\beta\in\mathbb{R}^{p}$.
\end{lemma}
The next lemma shows that the wedge estimator has an influence on the amount of sparsity needed for confidence sets. But as the simulations will suggest this might be improvable.
\begin{lemma}\label{wedge2}
 For the wedge penalty we have $C_{S_{\star}}=\sqrt{|S_{\star}|+1}$.
\end{lemma}

\subsection{Group Wedge}
This is a new idea for a more general wedge norm. It is based on the concept of grouping variables together.
Assume that we have $g$ disjoint groups $\{G_{1},...,G_{g}\}=\mathcal{G}$ with $\cup_{i=1}^{g}G_{i}=\{1,...,p\}$. Let us denote for a vector $\beta\in\mathbb{R}^{p}$ the $\ell_{2}$-norm on
a given group $G_{j}$ as $\lt{\beta_{G_{j}}}:=\sqrt{|G_{j}|}\sqrt{\sum_{i\in G_{j}}\beta_{i}^{2}}$. Then for a vector $\beta$ we define the following $g$-dimensional vector 
$$\beta^{\mathcal{G}}:=(\lt{\beta_{G_{1}}},\lt{\beta_{G_{2}}},...,\lt{\beta_{G_{g}}})^{T}.$$ 
Now we are able to define the group wedge in terms of the previously defined $g$-dimensional wedge norm on $\mathbb{R}^{g}$ as
$$\lVert\beta \rVert_{GWedge}:= \lVert \beta^{\mathcal{G}} \rVert_{W}.$$
We recover the wedge penalty again if we set the groups to be
$G_{i}:=\{i\}\text{ for any } i \in \{1,...,p\}$.
The first lemma shows that we have a norm again, the proof can be found in Section \ref{chap3:proofs}.
\begin{lemma}\label{grwedge2}
The group Wedge is in fact a norm. 
\end{lemma}

\begin{lemma}\label{grwedge1}
 The active sets are of the form $S=\cup_{i\in S_{g}}G_{i}$ for some subset of group indices $S_{g}\subset\{1,...,g\}$, and we have 
 $$\Omega^{S^{c}}(\beta_{S^{c}})=\lVert(\lt{\beta_{G_{s+1}}},...,\lt{\beta_{G_{g}}})^{T}\rVert_{W}.$$
 Moreover, the lower bounding gauge norm is the the Group LASSO norm with wedge groups
 $g(\beta)=\sum_{i=1}^{g}\lt{\beta_{G_{g}}}.$
\end{lemma}

\begin{lemma}\label{grwedge3}
 For the group wedge penalty we have $C_{S_{\star}}=\sqrt{|S_{\star}|+1}$, where $S_{\star}$ denotes the oracle set. 
\end{lemma}

\subsection{Lorentz norm}
Let us first define the Lorentz Cone (also known as the Ice Cream Cone):
$$\mathcal{A}:=\left\{\begin{pmatrix}a_{1}\\\smash{\vdots}\\a_{p-1}\\a_{p}\end{pmatrix}\in\mathbb{R}^{p}_{++}\middle| a_{p}\geq \lt{\begin{pmatrix}a_{1}\\\smash{\vdots}\\a_{p-1}\end{pmatrix}}\right\}$$

In a similar fashion to the definition of the wedge norm, the Lorentz norm is
$\lVert \beta\rVert_{Lo}:=\frac{1}{2}\min\limits_{a\in\mathcal{A}}\sum\limits_{i=1}^{p}\left(\frac{\beta^{2}_{i}}{a_{i}}+a_{i}\right)$.
This next lemma shows, that the Lorentz norm lets the index $p$ always be part of the preferred active sets.
\begin{lemma}\label{lorentz1}
 For the Lorentz norm it holds true that all the allowed sets contain $p$ and are of the form
 $$S=\left\{p, ...\text{ any combination of other variables}\right\}.$$
\end{lemma}

And we get the next lemma.
\begin{lemma}\label{lorentz2}
 For the Lorentz norm $g(\cdot)=\lo{\cdot}$.
\end{lemma}

\begin{lemma}\label{lorentz3}
 For the Lorentz norm $C_{S_{\star}}=3/2$.
\end{lemma}

The Lorentz norm can be generalized to include any set $P\subset\{1,..,p\}$ in the allowed sets.
The generalized convex cone is
$$\mathcal{B}:=\left\{b\in\mathbb{R}^{p}_{++}\middle| b_{j}\geq \lt{b_{P^{c}}}\text{ } \forall j\in P \right\},$$
and the generalized Lorentz norm can be defined as
$$\lVert \beta\rVert_{genLo}:=\frac{1}{2}\min\limits_{b\in\mathcal{B}}\sum\limits_{i=1}^{p}\left(\frac{\beta^{2}_{i}}{b_{i}}+b_{i}\right).$$
Now by an analogous proof to the proof of Lemma \ref{lorentz1}, we can see that the allowed sets of the generalized Lorentz norm always contain the set $P$. In particular an allowed set $S$
is of the form $S=P\cup B$, with $B\subset P^{c}$ being any subset of the complement of $P$.
The gauge function does not change, it is the $\ell_{1}$-norm and we still get a constant of 
$C_{S_{\star}}=(|P|+2)/2.$

\section{Simulations}
We look at the following linear model:
$Y=X\beta^{0} +\epsilon,$
where we have $n=100$ observations and $p=150$ variables with $\epsilon\sim\mathcal{N}(0,I)$. 
The design $X$ is randomly chosen, such that the covariance matrix has the following Toeplitz structure $\Sigma_{i,j}=0.9^{|i-j|}$.
The underlying parameter vector $\beta^{0}$ is chosen to be the regularly decreasing sequence
$$\beta_{\{1,...,s_{0}\}}:=\left(4,4-\frac{2}{s_{0}-1},4-2\cdot\frac{2}{s_{0}-1}, ... ,4-(s_{0}-2)\cdot\frac{2}{s_{0}-1},2\right)^{T},$$
where $s_{0}=|S_{0}|$ will be different values.
This structure of active set fits nicely in the wedge framework. Therefore to find a solution for the unknown $\beta^{0}$ we use the wedge
$$\hat{\beta}_{Wedge}=\amin_{\beta \in \mathbb{R}^{p}} \left\{ \ltn{Y-X\beta }+
\lambda \min_{a\in \mathcal{A}}\frac{1}{2}\sum\limits_{j=1}^{p}\left(\frac{\beta_{j}^{2}}{a_{j}}+a_{j}\right)\right\}.$$
Now we will construct confidence sets based on the two frameworks for the point-wise sets $\{1\}$,$\{2\}$,...,$\{p\}$. For each of these $p$ sets we compute $r=100$ repetitions.
To find the solution of the LASSO ($\ell_{1}$ is the gauge function in this case), the glmnet R package \citet{glmnet} has been used.
To solve the wedge the same code as in \citet{mi2} has been used. 
The following two cases have been considered: $s_{0}=5$ and $s_{0}=18$. Let us remark that $n/\log(p)\approx 20$ and $(n/\log(p))^{2/3}\approx 7$.
For each case the average coverage out of these 100 replications has been computed together with the average confidence set length.
The penalty level for the node-wise LASSO and node-wise Wedge have been chosen such that the average coverage are about the same, in order to compare their average set lengths.
Of course a reasonable penalty level for practical applications is up for debate.\\
\underline{Sparsity $s_{0}=5$:} 
For a very sparse setting the simulations, as seen in Figure \ref{simi11}, show that there is no essential difference between using the node-wise LASSO or the node-wise Wedge in order to construct the 
estimate of the precision matrix.

\begin{figure}[ht!]
\hspace*{-1cm}
\centering
  \includegraphics[width = 4.4cm]{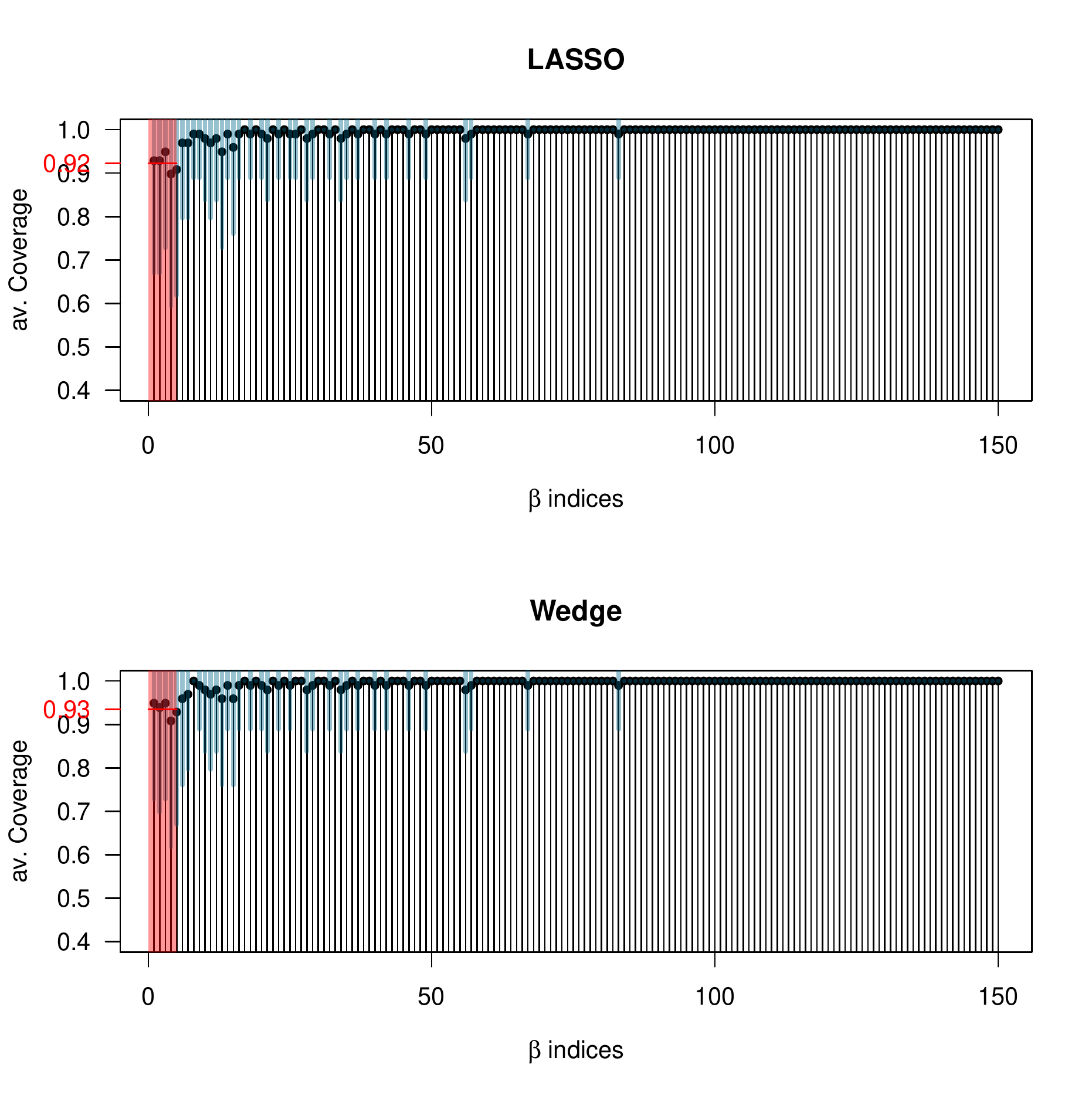}\includegraphics[width = 4.4cm]{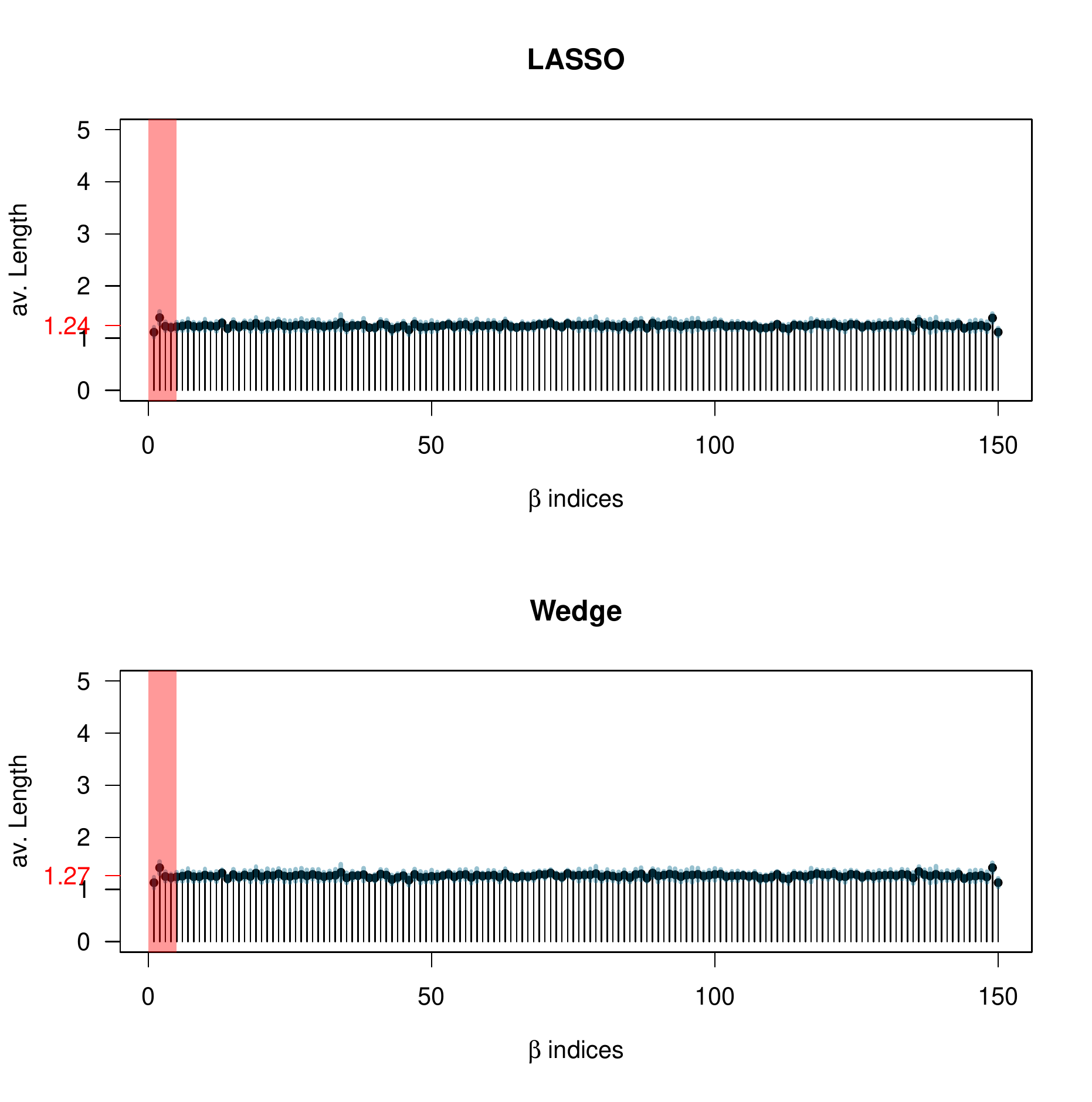}
  \hspace*{-1.5cm}
\caption{Left: Average coverage, Right: Average length, $s_{0}=5$, $\lambda_{LASSO}=15.5$, $\lambda_{Wedge}=15$ and in red are the mean values over all point-wise sets of the active set $S_{0}$.}
\label{simi11}
\end{figure}

\underline{Sparsity $s_{0}=18$:} 
Surprisingly for a less sparse setting the simulations, see Figure \ref{simi21}, still show no noticeable difference between the node-wise LASSO or the node-wise Wedge. This might indicate that
there could be a more direct way bound the estimation error expressed in the $\Omega$ norm, and that the bound of the remainder term
$\sqrt{n}\lambda C_{S_{\star}}\Upsilon_{S_{\star}}(\hat{\beta}_{J}-\beta^{0}_{J})$ might not be optimal for the wedge.
\begin{figure}[ht!]
\hspace*{-1cm}
\centering
  \includegraphics[width = 4.4cm]{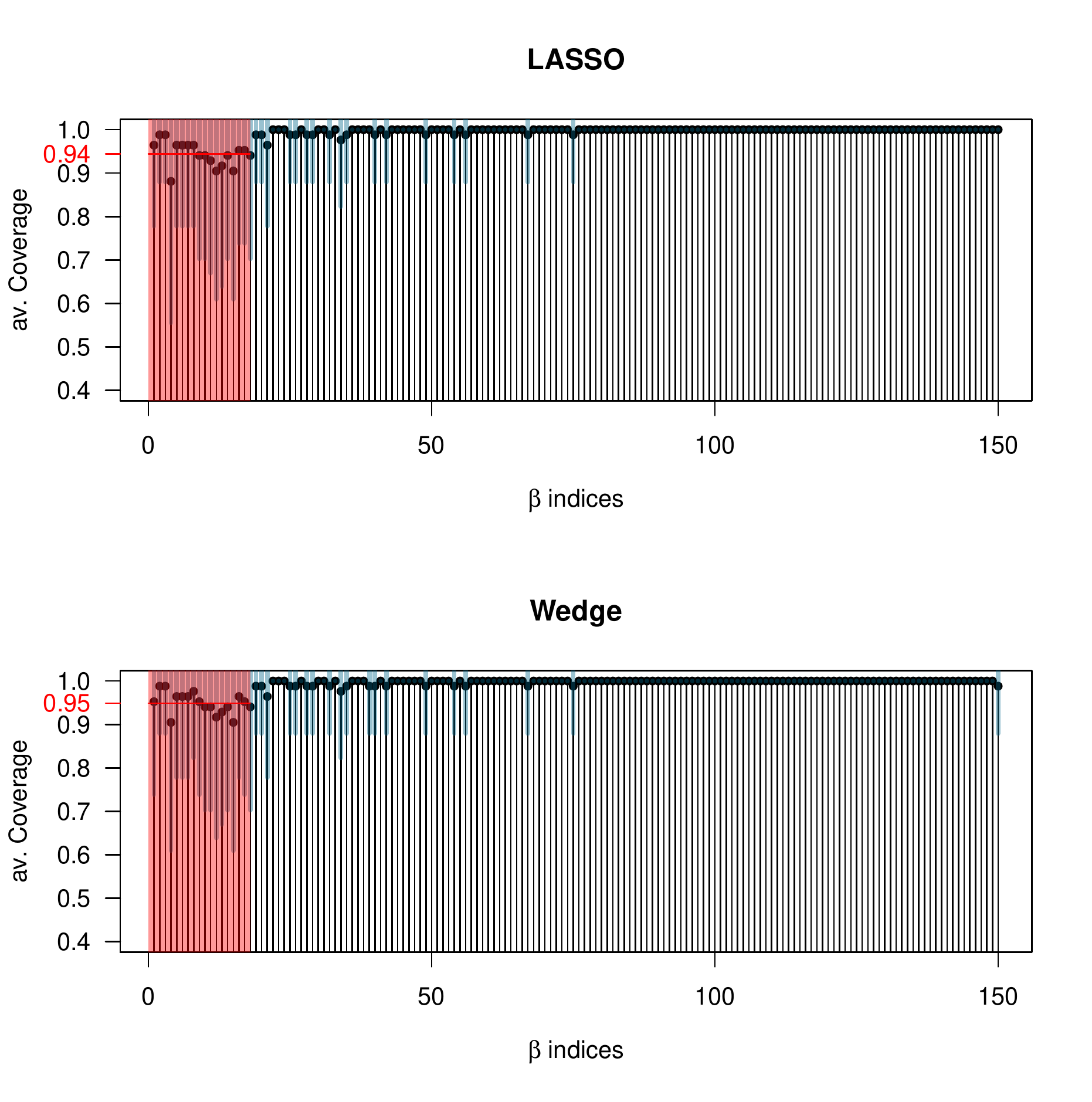}\includegraphics[width = 4.4cm]{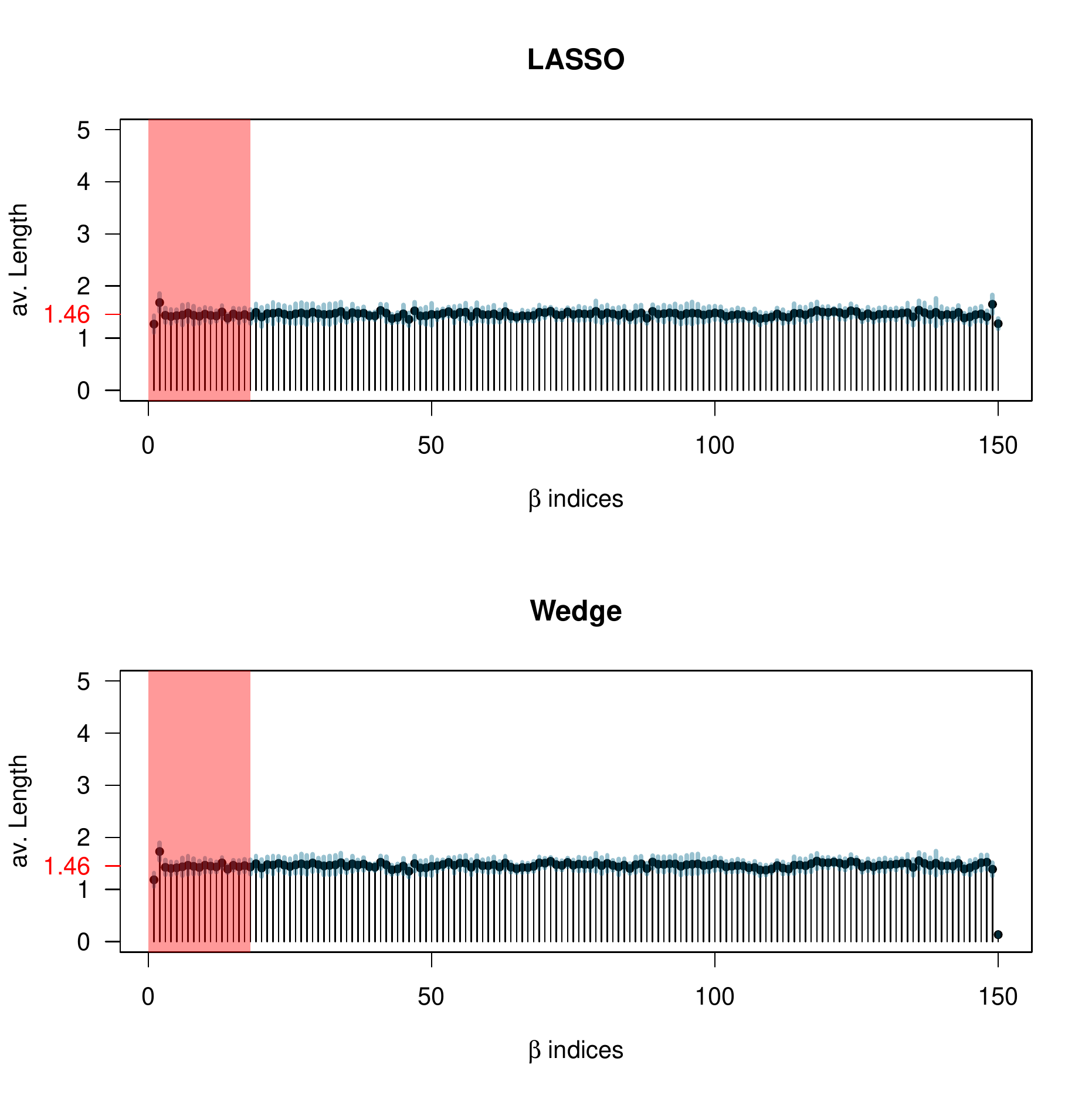}
    \hspace*{-1.5cm}
\caption{Left: Average coverage, Right: Average Length, $s_{0}=18$, $\lambda_{LASSO}=12$, $\lambda_{Wedge}=10$,  and in red are the mean values over all point-wise sets of the active set $S_{0}$.}
\label{simi21}
\end{figure}

\section{Conclusion}
Two frameworks for penalized estimators which incorporate structured sparsity patterns have been proposed. The first framework makes use of the gauge function,
which is in most cases an $\ell_{1}$ type norm due to the additivity of the lower bounding weak decomposable norms. The second framework is penalized by the structured sparse norm itself.
They are both quite general in the sense that they can be used in case of any weakly decomposable norm penalty, but they have their own properties regarding sparsity assumptions.
Interestingly the simulations suggest that at least for the presented Toeplitz case both frameworks seem to perform nearly indistinguishable, even for less strict sparsity assumptions.
Therefore it would be very interesting for future research to further understand if oracle results for the estimation error expressed in the weakly decomposable norms can be achieved.

\section{Proofs}\label{chap3:proofs}
In the dual world inequalities for norms change the direction.
\begin{lemma}\label{lem1:dwd}
 Let $\Omega(\cdot)$ and $\Upsilon(\cdot)$ be any two norms on $\mathbb{R}^{p}$ satisfying 
 $\Upsilon(\beta)\leq \Omega(\beta), \forall \beta\in \mathbb{R}^{p}.$
 Then for the corresponding dual norms we have the following inequality:
 $$\Upsilon^{*}(\omega)\geq\Omega^{*}(\omega), \forall \omega\in \mathbb{R}^{p}.$$
\end{lemma}
\begin{proof}
 First, let us remark that the unit balls $B_{\Upsilon}:= \{\beta: \Upsilon(\beta)\leq 1\}$ and $B_{\Omega}:= \{\beta: \Omega(\beta)\leq 1\}$ fulfill the following
 $$B_{\Upsilon}\supset B_{\Omega}.$$
 This is due to the fact that for all $\beta\in B_{\Omega}$ we have
 $$\Upsilon(\beta)\leq\Omega(\beta)\leq 1,$$
 which means that such $\beta$ are also element of the $B_{\Upsilon}$-Ball.
 Now if we look at the definition of the dual norm together with the fact that the supremum over the set $B_{\Upsilon}$ can only be bigger than over the set $B_{\Omega}$, we get
 $$\Upsilon^{*}(\omega)=\sup_{\beta\in B_{\Upsilon}}\omega^{T}\beta\geq\sup_{\beta\in B_{\Omega}}\omega^{T}\beta = \Omega^{*}(\omega), \forall \omega\in \mathbb{R}^{p}.$$
\end{proof}

Due to the disjoint nature of the definition of $\Upsilon_{S}$ as a sum of two norms on the set $S$ and set $S^{c}$ we can get an explicit formula for the dual norm.
\begin{lemma}\label{lem:wd}
For the weakly decomposable norm it holds true that 
 $$\Upsilon^{*}_{S}(\beta)=\max\left(\Omega^{*}(\beta_{S}),\Omega_{S^{c}}^{*}(\beta_{S^{c}})\right), \forall \beta\in \mathbb{R}^{p}.$$
\end{lemma}
\begin{proof}Let us show how to lower and upper bound it.\\
 $\text{\underline{Inequality 1:} }$ To show "$\geq$":
 \begin{align*}
  \Upsilon_{S}^{*}(\omega)&:=\sup_{\Upsilon_{S}(\beta)=1}\beta^{T}\omega \geq \sup_{\substack{\Upsilon_{S}(\beta)=1,\\ \beta=\beta_{S}}}\beta^{T}\omega\\
  &=\sup_{\Omega(\beta_{S})=1}\beta^{T}\omega\geq \Omega^{*}(\omega_{S}) 
 \end{align*}
 A similar result holds true if we restrict $\beta$ to $\beta_{S^{c}}$. Therefore the maximum lower bounds the dual of the $\Upsilon_{S}$-norm.\\
 $\text{\underline{Inequality 2:} }$ To show "$\leq$":
 \begin{align*}
  \Upsilon_{S}^{*}(\omega)&=\sup_{\Upsilon_{S}(\beta)=1}\beta^{T}\omega= \sup_{\Upsilon_{S}(\beta)=1}\left\{\beta_{S}^{T}\omega_{S}+\beta_{S^{c}}^{T}\omega_{S^{c}}\right\}\\
  &= \sup_{\Upsilon_{S}(\beta)=1}\left\{\frac{\beta_{S}^{T}\omega_{S}\Omega(\beta_{S})}{\Omega(\beta_{S})}+\frac{\beta_{S^{c}}^{T}\omega_{S^{c}}\Omega^{S^{c}}(\beta_{S^{c}})}{\Omega^{S^{c}}(\beta_{S^{c}})}\right\}\\
  &\leq \sup_{\Upsilon_{S}(\beta)=1}\left\{ \Omega^{*}(\omega_{S})\Omega(\beta_{S})+\Omega^{S^{c}}_{*}(\omega_{{S^{c}}})\Omega^{S^{c}}(\beta_{S^{c}})\right\}\\
  &=\sup_{a+b=1}\left\{ \Omega^{*}(\omega_{S})a+\Omega^{S^{c}}_{*}(\omega_{{S^{c}}})b\right\}\\
  &=\max\left(\Omega^{*}(\omega_{S}),\Omega^{S^{c}}_{*}(\omega_{{S^{c}}})\right)
  \end{align*}

\end{proof}

\begin{proof}[Proof of Lemma \ref{glemma}.]$\text{ }$\linebreak
\vspace{-5mm}
\begin{enumerate}
 \item The gauge function is again a norm on $\mathbb{R}^{p}$, because $0$ is in the convex set, see for example \citet{clarke} Theorem 2.36.
 \item First of all, the unit ball $B_{g}$ of norm $g$ contains all the unit balls $B_{\Upsilon_{S}}$ and $\operatorname{flip}_{J}(B_{\Upsilon_{S}})$, therefore
$$g(\beta)\leq \Upsilon_{S}(\beta)\text{ } \text{ } \forall \text{ } S \text{ allowed and}$$
$$g(\beta)\leq \Upsilon_{S}(\beta_{f(J)})\text{ } \text{ } \forall \text{ } S \text{ allowed.}$$
\item To prove that the dual norm of $g$ is
the maximum we need to make the following observations. First, from Lemma \ref{lem1:dwd} together with Lemma \ref{glemma} (2) we have
\begin{align*}
\Upsilon_{S}^{*}(\beta)&\leq g^{*}(\beta)\text{ } \text{ } \forall \text{ } S \text{ allowed.}
\end{align*}
Due to the fact that this holds for all allowed sets $S$ we get
$$\max\left(\max\limits_{S \text{ all.}}\Upsilon^{*}_{S}(\beta),\text{ }\max\limits_{S \text{ all.}}\Upsilon^{*}_{S}(\beta_{f(J)})\right)\leq g^{*}(\beta).$$
To prove the other inequality we need to look at the definition of the dual norm $g^{*}$:
$$g^{*}(z)=\max\limits_{g(x)\leq 1}z^{T}x.$$
Because the convex hull in the definition of $g$ is the set of all convex combinations of points in $\overline{B}\cup\operatorname{flip}_{J}(\overline{B})$ we have that
\begin{align*}
 x\in &\{y: g(y)\leq 1\}\Leftrightarrow x\in B_{g}\\
        \Leftrightarrow& \text{There exists some }n\in\mathbb{R}^{p} \text{ and } \sum_{i}^{n}\alpha_{i}=1, \alpha_{i}\geq 0,\\
        &\text{with a sequence of }b_{i}\in \overline{B}\cup\operatorname{flip}_{J}(\overline{B}), \\
        &\text{such that we can write }x=\sum_{i=1}^{n}\alpha_{i} b_{i}.
\end{align*}
That is why we can write
\begin{align*}
 g^{*}(z)&=\max\limits_{\substack{n\in \mathbb{N},\sum_{i}^{n}\alpha_{i}=1,\\b_{i}\in \overline{B}\cup\operatorname{flip}_{J}(\overline{B})}}\sum_{i=1}^{n}\left(\alpha_{i}z^{T}b_{i}\right)\\
 &\leq \max\limits_{n\in \mathbb{N},\sum_{i}^{n}\alpha_{i}=1}\sum_{i=1}^{n}\left(\alpha_{i}\cdot \max\limits_{b\in \overline{B}\cup\operatorname{flip}_{J}(\overline{B})}z^{T}b\right)\\
 &\leq 1\cdot \max\left(\max\limits_{S\text{ all.}}\Upsilon_{S}^{*}(z),\text{ }\max\limits_{S\text{ all.}}\Upsilon_{S}^{*}(z_{f(J)})\right).
\end{align*}
Thus equality holds, and one can easily see that\\
$\max\left(\max\limits_{S\text{ all.}}\Upsilon_{S}^{*}(z),\text{ }\max\limits_{S\text{ all.}}\Upsilon_{S}^{*}(z_{f(J)})\right)$ is a norm again.
Now the condition $\Omega(\beta_{f(J)})=\Omega(\beta)$ for all $\beta\in\mathbb{R}^{p}$ forces all $\Upsilon_{S}$-norms to have the same symmetrical property, 
and thus $B_{\Upsilon_{S}}=\operatorname{flip}_{J}(B_{\Upsilon_{S}})$. Therefore the one maximum can be omitted and the claim is proven.
For the characterization of the dual of the $\Upsilon_{S}$-norm we can just apply Lemma \ref{lem:wd}.

\item The function $\operatorname{flip}_{J}$ leads to
$g(\beta_{J}-\beta_{J^{c}})=g(\beta).$
Hence
\begin{align*}
 g(\beta_{J^{c}})&=g(\frac{1}{2}(\beta-\beta_{J})+\frac{1}{2}\beta_{J^{c}})\\
 &\leq \frac{1}{2}g(\beta)+\frac{1}{2}g(\beta_{J}-\beta_{J^{c}})\\
 &=\frac{1}{2}g(\beta)+\frac{1}{2}g(\beta)\\
 &=g(\beta)
\end{align*}

\end{enumerate}

\end{proof}

\begin{proof}[Proof of Theorem \ref{th1}]
Then
\begin{align*}
M(\hat{b}_{J}-\bh_{J})&=\sqrt{n}\hat{\Sigma}_{J}^{-1/2}(X_{J}-X_{J^{c}}\hat{B}_{J})^{T}X_{J}/n\cdot ...\\
&\qquad \cdot\left(T_{J}^{-1}(X_{J}-X_{J^{c}}\hat{B}_{J})^{T}(X\beta^{0}-X\bh+\epsilon )\right)\\
&= \sqrt{n}\hat{\Sigma}_{J}^{-1/2}(X_{J}-X_{J^{c}}\hat{B}_{J})^{T}\left(\epsilon +X(\beta^{0}-\bh)\right)\\
&=\hat{\Sigma}_{J}^{-1/2}(X_{J}-X_{J^{c}}\hat{B}_{J})^{T}\cdot ...\\
&\qquad\cdot \left(X_{J^{c}}(\beta^{0}-\bh)_{J}+X_{J^{c}}(\beta^{0}-\bh)_{J^{c}}+\epsilon \right)/\sqrt{n}\\
&=\hat{\Sigma}_{J}^{-1/2}(X_{J}-X_{J^{c}}\hat{B}_{J})^{T}X_{J^{c}}(\beta^{0}-\bh)_{J^{c}}/\sqrt{n}+...\\
&\qquad+\hat{\Sigma}_{J}^{-1/2}(X_{J}-X_{J^{c}}\hat{B}_{J})^{T}\epsilon/\sqrt{n}+...\\
&\qquad-\hat{\Sigma}_{J}^{-1/2}(X_{J}-X_{J^{c}}\hat{B}_{J})^{T}X_{J}(\bh-\beta^{0})_{J}/\sqrt{n}
\end{align*}
We can simplify the term
\begin{align*}
\hat{\Sigma}_{J}^{-1/2}(X_{J}-X_{J^{c}}\hat{B}_{J})^{T}X_{J}(\bh-\beta^{0})_{J}/\sqrt{n}=M(\bh_{J}-\beta^{0}_{J})
\end{align*}

That is why we can conclude that
\begin{align*}
M(\hat{b}_{J}-\beta^{0}_{J})&= M(\hat{b}_{J}-\bh_{J})+M(\bh_{J}-\beta^{0}_{J})\\
&=\hat{\Sigma}_{J}^{-1/2}(X_{J}-X_{J^{c}}\hat{B}_{J})^{T}\epsilon/\sqrt{n}+...\\
&\qquad+\sqrt{n}\hat{\Sigma}_{J}^{-1/2}(X_{J}-X_{J^{c}}\hat{B}_{J})^{T}X_{J^{c}}(\beta^{0}-\bh)_{J^{c}}/n \\
&=\underbrace{\hat{\Sigma}_{J}^{-1/2}(X_{J}-X_{J^{c}}\hat{B}_{J})^{T}\epsilon/\sqrt{n}}_{\text{Gaussian Random Variable}}+\underbrace{\lambda Z (\beta^{0}-\bh)_{J^{c}}/\sqrt{n}}_{\text{Remainder Term}}
\end{align*}
where $Z$ comes from the KKT conditions which fulfills:
$$\Psi^{*}(Z)\leq 1$$
$$\operatorname{tr}(Z^{T}\hat{B}_{J})=\Psi(\hat{B}_{J}).$$
The remainder term can be bounded with the generalized Cauchy Schwartz inequality in the $\ell_{\infty}$-norm by

\begin{align*}
\lambda \linf{Z(\beta^{0}-\bh)_{J^{c}}}/\sqrt{n}&= \lambda \max_{1\leq j\leq |J|} Z_{j}(\beta^{0}-\bh)_{J^{c}}/\sqrt{n}\\
&\leq \lambda\max_{1\leq j \leq |J| }g^{*}(Z_{j})g(\beta^{0}_{J^{c}}-\bh_{J^{c}})/\sqrt{n}\\
&\leq \lambda\Psi^{*}(Z)g(\beta^{0}_{J^{c}}-\bh_{J^{c}})/\sqrt{n} \\
&\leq \lambda g(\beta^{0}_{J^{c}}-\bh_{J^{c}})/\sqrt{n}\text{ }\text{ }\text{\scriptsize(KKT conditions.)}
\end{align*}
Dividing everything by $\sigma_{0}$ leads to the result.
\end{proof}

\begin{proof}[Proof of Lemma \ref{th21}]
Let us first observe:
\begin{align*}
 \Omega(\beta^{0}-\bh)&\leq \Omega(\beta^{0}_{S_{\star}}-\bh_{S_{\star}})+\Omega(\beta^{0}_{S^{*^c}}-\bh_{S^{*^c}})\\
                      &\leq \Omega(\beta^{0}_{S_{\star}}-\bh_{S_{\star}})+\Omega^{S^{*^c}}(\beta^{0}_{S^{*^c}}-\bh_{S^{*^c}})\\
                      &\qquad-\Omega^{S^{*^c}}(\beta^{0}_{S^{*^c}}-\bh_{S^{*^c}})+\Omega(\beta^{0}_{S^{*^c}}-\bh_{S^{*^c}})\\
                      &\leq \Upsilon_{S_{\star}}(\beta^{0}-\bh)+\Delta_{S^{*^c}}(\beta^{0}-\bh).
\end{align*}
Here we define $\Delta_{S^{*^c}}(\beta):=\Omega(\beta_{S^{*^c}})-\Omega^{S^{*^c}}(\beta_{S^{*^c}})$.
But we are left with another problem, how to bound $\Delta_{S^{*^c}}$. Understanding this distance will give us a bound on how far 
apart the weakly decomposable norm and the norm from the triangle inequality are.
Now let us take the optimal constant $C_{S_{\star}}$, which may depend on the active 
set of the oracle $|S_{\star}|$, such that
$\Omega(\beta_{S^{*^c}})\leq C_{S_{\star}}\cdot\Omega^{S^{*^c}}(\beta_{S^{*^c}}),\text{ }\forall \beta\in\mathbb{R}^{p}.$
Then with this we can write
$$\Delta_{S^{*^c}}(\beta^{0}-\bh)\leq (C_{S_{\star}}-1)\Omega^{S^{*^c}}(\beta^{0}_{S^{*^c}}-\bh_{S^{*^c}})\leq (C_{S_{\star}}-1)\Upsilon_{S_{\star}}(\beta^{0}-\bh).$$
In the last inequality we have used the weak decomposability condition. Therefore $\Omega(\beta^{0}-\bh)\leq C_{S_{\star}} \Upsilon_{S_{\star}}(\beta^{0}-\bh).$
\end{proof}

\begin{proof}[Proof of Theorem \ref{3th2}]
The first part follows directly the proof of Theorem \ref{th1}, with $\Xi$-norm instead of $\Psi$-norm.
The remainder term can be bounded with the generalized Cauchy Schwartz inequality in the $\ell_{\infty}$-norm by
\begin{align*}
\lambda \linf{Z(\beta^{0}-\bh)_{J^{c}}}/\sqrt{n}&= \lambda \max_{1\leq j\leq |J|} Z_{j}(\beta^{0}-\bh)_{J^{c}}/\sqrt{n}\\
&\leq \lambda\max_{1\leq j \leq |J|}\Omega^{*}(Z_{j})\Omega(\beta^{0}_{J^{c}}-\bh_{J^{c}})/\sqrt{n}\\
&\leq \lambda\Xi^{*}(Z)\Omega(\beta^{0}_{J^{c}}-\bh_{J^{c}})/\sqrt{n} \\
&\leq \lambda \Omega(\beta^{0}_{J^{c}}-\bh_{J^{c}})/\sqrt{n}\text{ }\text{ }\text{\scriptsize(KKT conditions.)}\\
&\leq \lambda 2\Omega(\beta^{0}-\bh)
\end{align*}
The last inequality comes directly from the weak decomposability of the allowed set 
$J$: 
\begin{align*}
 \Omega(\beta^{0}_{J^{c}}-\bh_{J^{c}})&\leq \Omega\left((\beta^{0}_{J^{c}}-\bh_{J^{c}}) +(\beta^{0}_{J}-\bh_{J})-(\beta^{0}_{J}-\bh_{J})\right)\\
  &\leq \Omega(\beta^{0}-\bh)+\Omega(\beta_{J}^{0}-\bh_{J})\\
  &\leq 2\Omega(\beta^{0}-\bh).
 \end{align*}
 Now with the calculation in the proof of Lemma \ref{th21} and by dividing everything by $\sigma_{0}$ the proof is finished.
 \end{proof}

\begin{proof}[Proof of Lemma \ref{lslope}]
 First of all, let us see that $l_{p}\lo{\beta}$ indeed is a lower bound for all weakly decomposable norms of $\Omega = J_{l}$.
 From \citet{benji1} we know that for any subset $S\subset\{1,...,p\}$ we have 
 $$\Upsilon_{S}(\beta)=\sum_{j=1}^{|S|}l_{j}|\beta|_{(j,S)}+\sum_{i=1}^{|S^{c}|}l_{|S|+i}|\beta|_{(i,S^{c})}$$
 with $1\geq l_{1}\geq l_{2}\geq...\geq l_{p}>0$ and $|\beta |_{(1,S^{c})} \ge \cdots \ge| \beta |_{(r,S^{c})}$ being the ordered sequence
in $\{\beta_{i}: i \in S^{c}\}$. We can now lower bound each $l_{i}$ and $l_{j}$ by the minimum of the decreasing sequence, namely $l_{p}$. That is why we get the sought lower bound
 $$l_{p}\lo{\beta}\leq \Upsilon_{S}(\beta)\leq \Omega(\beta) \text{ }\text{ }\text{ }\forall S\subset\{1,...,p\} \text{ and all }\beta\in \mathbb{R}^{p}.$$
 Therefore $\lambda_{p}\lo{\beta}$ is a candidate for the gauge function, but we need to show that this norm is the best lower bounding norm.
 Assume by contradiction that there is another norm $g(\cdot)$ on $\mathbb{R}^{p}$ such that
 $$l_{p}\lo{\beta}\leq g(\beta) \leq \Upsilon_{S}(\beta) \text{ }\text{ }\text{ }\forall S\subset\{1,...,p\} \text{ and all }\beta\in \mathbb{R}^{p},$$
 $$\text{and that there exists } \gamma\in\mathbb{R}^{p} \text{ such that }l_{p}\lo{\gamma}< g(\gamma).$$
 Denote the $k$-th standard basis vectors in $\mathbb{R}^{p}$ as $e_{k}$. Where $e_{k}$ is the vector having a one at the $k$-th entry and zeroes otherwise. 
 Then $\gamma$ can be written in the standard basis as a combination of the standard basis vectors
 $$\gamma=v_{1}e_{1}+v_{2}e_{2}+...+v_{p}e_{p}.$$
 From the above assumption and the fact that the set without the $k$-th index $\{1, ..., p\}\smallsetminus \{k \}$, denoted briefly as $\smallsetminus\{k\}$, is an allowed set, we have 
 that for each standard basis vector $e_{k}$ the following needs to hold true 
 $$l_{p}\lo{e_{k}}\leq g(e_{k})\leq \Upsilon_{\smallsetminus\{k\}}(e_{k})\text{ }\text{ }\text{ }\forall k\in \{1,...,p\}.$$
 Inserting the values $\lo{e_{k}}=1$ and $\Upsilon_{\smallsetminus\{k\}}(e_{k})= l_{p} \text{ }\forall k\in\{1,..,p\}$ leads to
 $$l_{p}\leq g(e_{k})\leq l_{p}.$$
 Therefore we can conclude that $g(e_{k})=l_{p}$ for all $k\in\{1,...,p\}$.
 Now applying the triangle inequality tho $g$ we have
 $$g(\gamma)\leq |v_{1}|g(e_{1})+...+|v_{p}|g(e_{p})=(|v_{1}|+...+|v_{p}|)l_{p}.$$
 On the other hand we get $l_{p}\lo{\gamma}=l_{p}(|v_{1}|+...+|v_{p}|).$
 This now clearly contradicts our assumption because
 $l_{p}\lo{\gamma}\nless g(\gamma)\leq l_{p}\lo{\gamma}.$
\end{proof}

\begin{proof}[Proof of Lemma \ref{lslope2}]
 By $\Omega(\beta_{S})=\sum_{j=1}^{|S|}l_{j}|\beta|_{(j,S)}$ and upper bounding all $l_{j}, j=\{1,...,|S|\}$ we have
 $$\Omega(\beta_{S^{*^c}})/l_{1}\leq \lo{\beta_{S^{*^c}}}.$$ 
 In a similar fashion by $\Omega^{S^{c}}(\beta_{S^{c}})=\sum_{i=1}^{|S^{c}|}l_{|S|+i}|\beta|_{(i,S^{c})}$ and lower bounding all $l_{i}, i=\{|S|+1,...,p\}$, we get
 $$\Omega^{S^{*^c}}(\beta_{S^{*^c}})/l_{p}\geq \lo{\beta_{S^{*^c}}}.$$
 Combining these two inequalities leads to 
 $$\Omega^{S^{*^c}}(\beta_{S^{*^c}})l_{1}/l_{p}\geq \Omega(\beta_{S^{*^c}}).$$
 for the SLOPE penalty $C_{S_{\star}}= l_{1}/l_{p}=o(\log(p))$.
 The last equality comes from the Bonferroni $l$-sequence choice in \citet{candes1}.
\end{proof}

\begin{proof}[Proof of Lemma \ref{wedge1}]
 First we know by \citet{mi2} that $\lo{\beta}\leq\Upsilon(\beta)$ for all allowed sets $S$ and all $\beta\in\mathbb{R}^{p}$.
 Now in order to show that this is the best lower bounding norm, let us assume by contradiction that there exists another norm $g(\cdot)$ which is strictly 
 better than $\lo{\cdot}$:
 \begin{align*}
  \lo{\beta}&\leq g(\beta)\leq \Upsilon(\beta) \text{ }\forall S \text{ allowed }\forall \beta\in\mathbb{R}^{p},\\
  \exists \gamma \in\mathbb{R}^{p}\text{ such that }\lo{\gamma}&< g(\gamma)\leq \Upsilon(\gamma) \text{ }\forall S \text{ allowed}.
 \end{align*}
 Define the standard basis as $e_{k}, k\in\{1,...,p\}$ being the vector having a one at the $k$-th entry and zero entries otherwise.
 Let us fix any allowed set $S$. It is straight forward to check that
 $$\Upsilon(e_{1})=1,\text{ and }\Upsilon(e_{s+1})=1.$$
 By the assumption we get that
 $$1=\lo{e_{s+1}}\leq g(e_{s+1})\leq \Upsilon(e_{s+1})=1 \text{ }\forall S \text{ allowed.}$$
 And similarly for the first standard basis vector $e_{1}$ we have
 $$1=\lo{e_{1}}\leq g(e_{1})\leq \Upsilon(e_{1})=1 \text{ }\forall S \text{ allowed.}$$
 Now because $s\in\{1,...,p-1\}$ we get that:
 $$g(e_{k})=1, \text{ for any }k\in\{1,...,p\}.$$
 So we know the values that $g$ attains for the standard basis. With this we can conduct the following contradiction. The vector $\gamma$ has a unique representation in the standard basis 
 $\gamma=v_{1}e_{1}+v_{2}e_{2}+...+v_{p}e_{p}$,
 and therefore we can apply the triangle inequality $p$ times to get:
 \begin{align*}
   g(\gamma)&\leq |v_{1}|g(e_{1})+|v_{2}|g(e_{2})+...+|v_{p}|g(e_{p})\\
   &=|v_{1}|+|v_{2}|+...+|v_{p}|\\
   &=\lo{\gamma}
 \end{align*}
 This contradicts our assumption that $\lo{\gamma}<g(\gamma)$, and the claim is proven.
\end{proof}

\begin{proof}[Proof of Lemma \ref{wedge2}]
For any allowed set $S$, the weakly decomposable $\Upsilon$-norm consists of the following two parts
$$\Omega(\beta_{S^{c}})=\min_{a_{S^{c}}\in\mathcal{A}_{S^{c}}}\frac{1}{2}\bigg(\sum_{j\in S^{c}}\Big(\frac{\beta_{j}^{2}}{a_{j}}+a_{j}\Big)+s\cdot a_{s+1}\bigg),$$
$$\Omega^{S^{c}}(\beta_{S^{c}})=\min_{a_{S^{c}}\in\mathcal{A}_{S^{c}}}\frac{1}{2}\bigg(\sum_{j\in S^{c}}\frac{\beta_{j}^{2}}{a_{j}}+a_{j}\bigg).$$
Here we have used that $a_{j}\geq a_{j+1}$ for all $1\geq j\leq p-1$.
Because of the structure of the cone $\mathcal{A}$ we have
\begin{align*}
 \Omega(\beta_{S^{c}})&=\min\limits_{a_{S^{c}}\in\mathcal{A}_{S^{c}}}\frac{1}{2}\Big(\sum\limits_{j=s+2}^{p}\Big(\frac{\beta_{j}^{2}}{a_{j}}+a_{j}\Big)+\frac{\beta_{s+1}^{2}}{a_{s+1}}+(s+1)a_{s+1}\Big)\\
 &\leq \min\limits_{a_{S^{c}}\in\mathcal{A}_{S^{c}}}\frac{1}{2}\bigg(\sum\limits_{j=s+2}^{p}\Big(\frac{\beta_{j}^{2}}{a_{j}}+(s+1)a_{j}\Big)+\frac{\beta_{s+1}^{2}}{a_{s+1}}+(s+1)a_{s+1}\bigg)\\
 &=\sqrt{s+1}\min\limits_{a_{S^{c}}\in\mathcal{A}_{S^{c}}}\frac{1}{2}\bigg(\sum\limits_{j=s+1}^{p}\frac{\beta_{j}^{2}}{\sqrt{s+1}a_{j}}+\sqrt{s+1}a_{j}\bigg).
\end{align*}
In the second inequality we added $\sum_{j=s+2}^{p}s a_{j}\geq 0$, and in the last inequality we take $\sqrt{s+1}$ outside the minimum.
Now in this setting we know that for $a_{S^{c}}\in\mathcal{A}_{S^{c}}$ we have $a_{s+1}\geq a_{s+2}\geq ... \geq a_{p}\geq 0$. Furthermore 
$a^{'}_{S^{c}}:=(\sqrt{s+1}a_{s+1},\sqrt{s+1}a_{s+2},...,\sqrt{s+1}a_{p})^{T}\in \mathcal{A}_{S^{c}}$, in fact any sequence $a_{S^{c}}\in\mathcal{A}_{S^{c}}$ 
can be displayed by a sequence which is multiplied by $\sqrt{s+1}$. Therefore
\begin{align*}
 \Omega(\beta_{S^{c}})&\leq \sqrt{s+1}\min_{a^{'}_{S^{c}}\in\mathcal{A}_{S^{c}}}\frac{1}{2}\bigg(\sum_{j=s+1}^{p}\frac{\beta_{j}^{2}}{a^{'}_{j}}+a^{'}_{j}\bigg)\\
 &\leq\sqrt{s+1}\cdot\Omega^{S^{c}}(\beta_{S^{c}})
\end{align*}
\end{proof}

\begin{proof}[Proof of Lemma \ref{grwedge1}]
 The $\ell_{2}$-norm does not have any non trivial active sets, and the $g$-dimensional wedge norm has active sets $S=\{1,...,s\}$ for any $s\in\{1,...,g\}$. 
 Combining theses facts leads to the conclusion that only for active sets of the form $S=\cup_{i\in S_{g}}G_{i}$ we have weak decomposability:
 $$\lVert\beta_{S} \rVert_{grW}+\lVert\beta_{S^{c}} \rVert_{grW}\leq \lVert\beta \rVert_{grW}.$$
 Because of the definition of the group wedge as a composition of the wedge and $\ell_{2}$-norm this is the best lower bound.
 For the gauge function $g$ it is easy to see that by applying Lemma \ref{wedge1}, we get the Group LASSO.
\end{proof}

\begin{proof}[Proof of Lemma \ref{grwedge2}]$\text{ }$\\
\begin{enumerate}\vspace{-5mm}
 \item $\lVert\beta\rVert_{grW}=0\text{ }\Longleftrightarrow \text{ }\lt{\beta_{G_{i}}}=0\text{ } \forall i \in \{1,..,g\}\Longleftrightarrow\beta \equiv 0.$
 \item The following calculations hold true:
       \begin{align*}
        \lVert a\beta\rVert_{grW}&= \lVert(\lt{a\beta_{G_{s+1}}},...,\lt{a\beta_{G_{g}}})^{T}\rVert_{W}\\
        &= \lVert a(\lt{\beta_{G_{s+1}}},...,\lt{\beta_{G_{g}}})^{T}\rVert_{W}\\
        &= a \lVert \beta\rVert_{grW}.
       \end{align*}
 \item The triangle inequality holds due to the properties of the wedge and $\ell_{2}$-norms.
       \begin{align*}
        \lVert &\beta+\gamma\rVert_{grW}= \lVert(\lt{\beta_{G_{s+1}}+\gamma_{G_{s+1}}},...,\lt{\beta_{G_{g}}+\gamma_{G_{g}}})^{T}\rVert_{W}\\
        &\leq  \lVert(\lt{\beta_{G_{s+1}}}+\lt{\gamma_{G_{s+1}}},...,\lt{\beta_{G_{g}}}+\lt{\gamma_{G_{g}}})^{T}\rVert_{W}\\
        &\leq  \lVert\beta^{\mathcal{G}}+\gamma^{\mathcal{G}}\rVert_{W}\\
        &\leq  \lVert\beta^{\mathcal{G}}\rVert_{W}+\lVert\gamma^{\mathcal{G}}\rVert_{W} = \lVert \beta\rVert_{grW} + \lVert\gamma\rVert_{grW}
       \end{align*}

\end{enumerate}

\end{proof}

\begin{proof}[Proof of Lemma \ref{grwedge3}]
 By applying Lemma \ref{wedge2} in this context, together with $S_{\star}$ being the optimal active groups, we immediately get the desired result.
\end{proof}

\begin{proof}[Proof of Lemma \ref{lorentz1}]
  By \citet{sara1} we know that for the structured sparsity norms, as introduced in \citet{mi2}, it holds that
  $$S\text{ is an allowed set }\Longleftrightarrow \mathcal{A}_{S}:=\left\{a_{S}: a\in\mathcal{A}\right\}\subset\mathcal{A}.$$
  Let us distinguish two cases, in order to proof the lemma. \\
  \underline{Case 1:} Assume $p\notin S$.\\
  Therefore $A_{S}$ consists of vectors with the $p$-th variable set to zero. This means that there exists at least one vector $a$ such that $a_{S}$ is not in $\mathcal{A}$
  $$a_{p,S}=0\ngeq \lt{\begin{pmatrix}a_{1}\\\smash{\vdots}\\a_{p-1}\end{pmatrix}_{S}}.$$
  In other words $\mathcal{A}_{S}\nsubseteq \mathcal{A}$. Therefore sets $S$ which do not contain $p$ cannot be allowed sets.\\
  \underline{Case 2:} Assume that the set $S$ satisfies $S\ni p$.\\
  For each vector $a_{S}$ in $A_{S}$ we have
  $$a_{p,S}=a_{p}\geq \lt{\begin{pmatrix}a_{1}\\\smash{\vdots}\\a_{p-1}\end{pmatrix}}\geq \lt{\begin{pmatrix}a_{1}\\\smash{\vdots}\\a_{p-1}\end{pmatrix}_{S}}.$$
  The first inequality is due to $a$ being in $\mathcal{A}$. For the second inequality it suffices to see that the $\ell_{2}$ norm can only decrease by setting certain
  values to zero. therefore we know that any set $S$ which contains $p$ fulfills $\mathcal{A}_{S}\subset \mathcal{A}$.
\end{proof}

\begin{proof}[Proof of Lemma \ref{lorentz2}]
 Again by \citet{mi2} we know that $\lo{\beta}\leq\Upsilon(\beta)$ for all allowed sets $S$ and all $\beta\in\mathbb{R}^{p}$.
 Define the standard basis as $e_{k}, k\in\{1,...,p\}$. Let us fix any allowed set $S$ from Lemma \ref{lorentz1}. 
 We can calculate that
 $$\Upsilon(e_{k})=\sqrt{2} \text{ if }k\in S\smallsetminus\{p\}, \Upsilon(e_{k})=1 \text{ if }k\notin S\smallsetminus\{p\},$$
 $$\Upsilon(e_{p})=1.$$
 Taking the special allowed set $S=\{g\}$ we have
 $$1=\lo{e_{k}}\leq g(e_{k})\leq \Upsilon_{\{g\}}(e_{k})=1, \forall k\in\{1,..,p\}.$$
 This leads to $g(e_{k})=1 , \forall k\in\{1,..,p\}$. Therefore we can use the same idea of the proof from Lemma \ref{wedge1}, and we get that $g(\cdot)=\lo{\cdot}$.
\end{proof}

\begin{proof}[Proof of Lemma \ref{lorentz3}]
 We have that 
 $$\Omega^{S^{c}}(\beta_{S^{c}})=\min_{a_{S^{c}}\in\mathcal{A}_{S^{c}}}\frac{1}{2}\sum_{j\in S^{c}}\left(\frac{\beta^{2}_{j}}{a_{j}}+a_{j}\right)=\lo{\beta_{S^{c}}}.$$
 This is due to $p\notin S^{c}$ and therefore the $\{a_{j}: j\in S^{c}\}$ can be chosen independently of each other, leading to the minimum $a_{j}=\beta_{j}$.
 Furthermore we have the following upper bound:
 \begin{align*}
  \Omega(\beta_{S^{c}})&= \min_{a\in\mathcal{A}}\frac{1}{2}\sum_{j\in S^{c}}\left(\frac{\beta^{2}_{j}}{a_{j}}+a_{j}\right)+ a_{p}\\
  &\leq \min_{a\in\mathcal{A}}\frac{1}{2}\sum_{j\in S^{c}}\left(2\beta_{j}\right)+\lVert \beta_{S^{c}}\rVert_{2} \text{ with }a_{j}=\beta_{j}\\
  &= \lo{\beta_{S^{c}}}+\frac{1}{2}\lVert\beta_{S^{c}}\rVert_{2}\\
  &\leq \frac{3}{2}\lo{\beta_{S^{c}}}=\frac{3}{2}\Omega^{S^{c}}(\beta_{S^{c}}).
 \end{align*}
 Which leads to the desired constant.

\end{proof}

\section*{Appendix: A refined Sharp Oracle Inequality}
Let us first remind us of the definition of the theoretical lambda\\
$\lambda^{m}:=\max\left(\Omega^{*}(\epsilon^{T}X_{S_{\star}}),\Omega^{S_{\star}^{c}*}(\epsilon^{T}X_{S_{\star}^{c}})\right)/n=\Upsilon_{S_{\star}}^{*}\left(\epsilon^{T}X\right)/n.$
Lemma \ref{th11} refines the sharp oracle result from \citet{sara1}. In particular, the sharp oracle inequality from \citet{sara1} measures a variation of the estimation error in the following way
$\Omega(\bh_{S_{\star}}-\beta^{*})+\Omega^{S^{c}_{\star}*}(\bh_{S^{c}_{\star}}),$
with $S_{\star}=\supp(\beta^{*})$. Let us remark here that the optimal oracle parameter $\beta^{*}$ may not be equal to $\beta^{0}$. Therefore we have no guarantee to get an upper bound on the 
estimation error expressed as $\Upsilon_{S_{\star}}(\bh-\beta^{0})$. But this is needed for both confidence frameworks to work. 
Therefore we will rework Theorem 4.1 from \citet{sara1} to make $S$ and $\beta$ independent of each other.
Let us furthermore define the $\Omega$-effective sparsity as in \citet{sara1} and denote it $\Gamma^{2}_{\Omega}(L,S)$.
\begin{lemma}[Refined Sharp Oracle Inequality]\label{th11}
 Assume that  $0\leq \delta< 1$, and also that $\frac{\lambda^{m}}{\lambda}= c$ with $0<c<1$.
 We invoke weak decomposability for $S$ and $\Omega$. Here the active set $S$ and parameter vector $\beta$ can be chosen independently.
 Then it holds true that
%
  \begin{align}\label{3sharp}
  \Upsilon_{S_{\star}}(\bh-\beta^{0})&\leq \min_{\beta,S,\delta}\left(\Upsilon_{S}(\beta^{0}-\beta) +C_{1}\lambda\GO^{2}(L_{S},S)+\right. \nonumber\\
  &\qquad\qquad \left. +\frac{C_{2}}{\lambda}\ltn{X(\beta-\beta^{0})}+C_{3}\Omega(\beta_{S^{c}})\right),
  \end{align}
 with $L_{S}:= \frac{\lambda+\lambda^{m}}{\lambda-\lambda^{m}}\frac{1+\delta}{1-\delta}$, constants $C_{1}=\frac{[(1+\delta)(1+c)]^{2}}{2\delta(1-c)}$, 
 $C_{2}=\frac{1}{2\delta(1-c)}$, $C_{3}=\frac{2}{\delta(1-c)}$ and
 
\begin{align*}
 S_{\star}:=\amin_{S}\min_{\beta,\delta}\bigg[&\Upsilon_{S}(\beta^{0}-\beta) +C_{1}\lambda\GO^{2}(L_{S},S)+...\bigg.\\
 &\quad\bigg. +\frac{C_{2}}{\lambda}\ltn{X(\beta-\beta^{0})}+C_{3}\Omega(\beta_{S^{c}})\bigg].
\end{align*}
\end{lemma}

\begin{proof}
Follows directly from the proof of the main theorem in \cite{sara1} and the triangle inequality.
\end{proof}



\newpage



%
                             %
\bibliography{reference2}{}  %
\bibliographystyle{plainnat} %
                             %
\end{document}